\documentclass[11pt]{article}

\usepackage[T1]{fontenc}
\usepackage{lmodern}
\usepackage[letterpaper,margin=0.88in]{geometry}
\usepackage{microtype}

\usepackage{amsmath,amssymb,amsthm}
\usepackage{mleftright}
\mleftright

\usepackage{array,booktabs,multirow}
\usepackage{graphicx}

\usepackage{algorithm}
\usepackage{algpseudocode}

\usepackage[numbers,sort&compress]{natbib}
\usepackage{xurl}
\usepackage[hidelinks,pdfusetitle]{hyperref}
\usepackage[nameinlink,noabbrev]{cleveref}
\makeatletter
\providecommand{\theHALG@line}{}
\renewcommand{\theHALG@line}{\thealgorithm.\arabic{ALG@line}}
\newenvironment{breakablealgorithm}
  {\par\medskip\noindent\refstepcounter{algorithm}\hrule height.8pt depth0pt\kern2pt
   \renewcommand{\caption}[1]{{\raggedright\textbf{\fname@algorithm~\thealgorithm}~##1\par}\kern2pt\hrule\kern2pt}}
  {\kern2pt\hrule\par\medskip}
\makeatother

\graphicspath{{figures/}}

\theoremstyle{plain}
\newtheorem{theorem}{Theorem}[section]
\newtheorem{proposition}[theorem]{Proposition}
\newtheorem{lemma}[theorem]{Lemma}

\theoremstyle{definition}
\newtheorem{definition}[theorem]{Definition}
\newtheorem{example}[theorem]{Example}
\AddToHook{env/proposition/begin}{\crefalias{theorem}{proposition}}
\AddToHook{env/lemma/begin}{\crefalias{theorem}{lemma}}
\AddToHook{env/corollary/begin}{\crefalias{theorem}{corollary}}
\AddToHook{env/definition/begin}{\crefalias{theorem}{definition}}
\AddToHook{env/example/begin}{\crefalias{theorem}{example}}
\algrenewcommand\algorithmicrequire{\textbf{Input:}}
\algrenewcommand\algorithmicensure{\textbf{Output:}}
\algblockdefx[Switch]{Switch}{EndSwitch}[1]{\textbf{switch} #1 \textbf{do}}{\textbf{end switch}}
\algblockdefx[Case]{Case}{EndCase}[1]{\textbf{case} #1:}{\textbf{end case}}
\algtext*{EndCase}

\title{Polar Switch: Fast Muon Orthogonalization with Polynomial Circuits\thanks{Code is available at \url{https://github.com/zhehangdu/polar-switch}.}}
\author{Zhehang Du \qquad Weijie Su\thanks{University of Pennsylvania.
Emails: \texttt{\{duz,suw\}@wharton.upenn.edu}.}}
\hypersetup{pdfauthor={Zhehang Du and Weijie Su}}
\date{October 19, 2026}

\begin{document}
\maketitle

\begin{abstract}
Muon, a 2024-born optimizer for language models, differs from the 2015-born spectral gradient descent~\citep{carlson2015stochastic,carlson2015preconditioned} mainly in its fast, GPU-friendly Newton--Schulz polynomial for orthogonalizing the gradient matrix. Accordingly, a line of research proposes faster polynomial iterations for gradient orthogonalization, such as \texttt{Polar Express} and \texttt{Gram Newton--Schulz}, both built on quintic polynomials. In this paper, we introduce \texttt{Polar Switch}, a new approach to polynomial gradient orthogonalization. Using simple polynomial circuits, \texttt{Polar Switch} reaches a given high polynomial degree at a lower FLOP cost than quintic iterations. We define a circuit as a short sequence of symmetric matrix products on a few square buffers, and circuits with $k$ products form a restricted family of odd polynomials of degree up to $2^{k+1}+1$. We choose the sequence of products in each circuit, their coefficients, and how the circuits are composed by numerically minimizing a worst-case bound on the polar approximation error. This minimization is subject to a FLOP budget that accounts for symmetric matrix multiplication kernels and a rounding constraint for numerical stability in half precision. Our two recommended \texttt{Polar Switch} methods have 15.5--18.7\% lower runtime than \texttt{Polar Express} on square matrices and 8.6--16.2\% lower runtime than \texttt{Gram Newton--Schulz} on rectangular matrices with aspect ratios 3.25, 3.5, and 8, with comparable approximation accuracy. In pretraining Modded-NanoGPT \#91, Qwen3-0.6B, and Gemma 3-1B, \texttt{Polar Switch} reaches final validation losses similar to those of established orthogonalization methods. Across the more than 14 million matrices that each of the two recommended methods orthogonalized in these runs, the estimated output spectral norms never exceed the designed bounds by more than 3\%, demonstrating their numerical stability.

\end{abstract}

\section{Introduction}
\label{sec:intro}

For years, Adam~\citep{kingma2015adam} and AdamW~\citep{loshchilov2019decoupled}
have been the go-to optimizers for training large language models (LLMs)~\citep{brown2020language}. Both scale the momentum
coordinatewise, without explicitly accounting for the matrix structure of the parameters. Muon~\citep{jordan2024muon} takes a different approach in that it uses the matrix structure of the momentum and approximates its polar factor to form the update.
Studies have shown that Muon can reach the same training loss as AdamW with less
computation~\citep{liu2025muon,ai2025practical}, and it has since been used to train several frontier open-weight LLMs,
including Kimi K3~\citep{kimiteam2026kimi}, GLM-5~\citep{glm5team2026glm5}, Step 3.5 Flash~\citep{huang2026step35flash}, DeepSeek-V4~\citep{deepseekai2026deepseekv4}, and Qwen3.8-Flash-Next~\citep{qiu2026qwen38next}. This advantage comes with additional computation at each optimizer
step, because Muon must approximate the polar factor. The question is therefore how to compute
this approximation accurately and efficiently.

Let $W\in\mathbb{R}^{n\times m}$ be a matrix-shaped parameter, $\nabla W$ its gradient,
and $M$ its momentum matrix. Without loss of generality, let $n\le m$, because matrices with more rows than columns are handled by transposition.
With momentum coefficient $\mu$ and
learning rate $\eta$, a simplified Muon update is
\[
  M\leftarrow\mu M+\nabla W,
  \qquad
  W\leftarrow W-\eta\operatorname{polar}(M).
\]
If $M=U\Sigma V^\mathsf{T}$ is the compact singular value decomposition (SVD), then
$\operatorname{polar}(M)=UV^\mathsf{T}$~\citep{higham2008functions}.
However, computing an SVD at every optimizer step is expensive for LLM
training. Muon approximates the polar factor using
polynomial iterations inspired by Newton--Schulz~\citep{jordan2024muon}.
First, we normalize $M$ as
$X\leftarrow M/(\lVert M\rVert_F+\epsilon)$ with a small $\epsilon>0$ ($10^{-7}$ in our implementation), so that all singular values
of $X$ lie in $[0,1]$. At iteration $t$, we choose a quintic polynomial
$p_t(x)=c_tx+b_tx^3+a_tx^5$ with coefficients $a_t,b_t,c_t$ and compute
\begin{equation}
  A\leftarrow XX^\mathsf{T},
  \qquad
  Z\leftarrow b_tA+a_tA^2,
  \qquad
  X\leftarrow ZX+c_tX.
  \label{eq:one-step}
\end{equation}
Let $X=U\Sigma V^\mathsf{T}$ be the compact SVD of $X$.  The update is
$X\leftarrow U\left(c_t\Sigma+b_t\Sigma^3+a_t\Sigma^5\right)V^\mathsf{T}$.
Suppose that, at each iteration $t$, $p_t(\sigma)\ge0$ for every singular value
$\sigma$ of the current matrix $X$. The iteration replaces each singular value
$\sigma$ by $p_t(\sigma)$. After $T$ iterations, each initial singular value $\sigma$
becomes $(p_T\circ\cdots\circ p_1)(\sigma)$.
Approximating the polar factor in Muon therefore requires this composition to
approximate one on the positive initial singular values.

The classical Newton--Schulz iteration uses $p_t(x)=(3x-x^3)/2$, whose
repeated application converges to one for $0<x\le1$~\citep{higham2008functions}.
Muon uses Jordan's fixed quintic
$p_t(x)=3.4445x-4.775x^3+2.0315x^5$~\citep{jordan2024muon}.
\texttt{Polar Express} (\textsf{PE}) uses minimax approximation to choose a different quintic
polynomial at each iteration, minimizing its maximum deviation from one over the
current singular value interval~\citep{amsel2026polar}.
Jordan's method and \texttt{Polar Express} differ in their coefficients, but both use
quintic iterations. We therefore ask the following question:

\begin{center}
\begin{minipage}{0.88\linewidth}
\centering\itshape
Under a fixed FLOP budget for matrix orthogonalization, do repeated
quintic iterations yield the most accurate polar approximation?
\end{minipage}
\end{center}

We answer this question negatively with \texttt{Polar Switch}, which replaces repeated quintic iterations with a few high-degree polynomial circuits.
Each circuit is a short sequence of symmetric matrix products on a small number of square buffers.
For example, the first circuit of \textsf{Polar Switch 3-55} (\textsf{PS3-55}) uses three square buffers $B_1,B_2,B_3\in\mathbb{R}^{n\times n}$, with the coefficients $a_\tau,b_\tau,c_\tau$ given in \Cref{alg:polar-switch-methods}:
\begin{align}
 &B_1\leftarrow XX^\mathsf{T}, & &B_2\leftarrow a_1B_1^2+b_1B_1+c_1I, & &B_3\leftarrow a_2B_2^2+b_2B_2+c_2I, & &B_2\leftarrow a_3B_3^2+b_3B_3+c_3I,\nonumber\\
 & & &B_3\leftarrow a_4B_2^2+b_4B_3+c_4I, & &B_1\leftarrow a_5B_3^2+b_5B_1+c_5I, & &X\leftarrow B_1X.
 \label{eq:ps355}
\end{align}
Each squaring doubles the degree in $XX^\mathsf{T}$, to $2$, $4$, $8$, $16$, and $32$, so the circuit applies an odd polynomial of degree $65$ to the singular values.
\texttt{Polar Switch} searches numerically for the circuits and their coefficients that approximate the polar factor most accurately under a FLOP budget.
\Cref{tab:circuit-cost} compares \texttt{Polar Switch} with the quintic baselines.
\textsf{Polar Switch 2-55} (\textsf{PS2-55}) and \textsf{PS3-55} use fewer FLOPs than \texttt{Polar Express} and \texttt{Gram Newton--Schulz} and less memory than \texttt{Gram Newton--Schulz}.

\begin{table}[htbp]
\centering
\caption{Memory in matrix elements, FLOPs, and final singular value intervals for $X\in\mathbb{R}^{n\times m}$ with $n\le m$ and $\alpha=m/n$. In exact arithmetic, each method maps singular values in $[0.001,1]$ into its final interval. \textsf{Jordan} and \textsf{Gram Newton--Schulz} multiply the output of their final iteration by a constant that centers this interval at one. \Cref{alg:polar-switch-methods} gives the iterations of \texttt{Polar Switch}.\label{tab:circuit-cost}}
\begin{tabular}{@{}lccc@{}}
\toprule Method & Memory & FLOPs & Final interval \\ \midrule
\textsf{Jordan}~\citep{jordan2024muon} & $n^2(2\alpha+2)$ & $(15\alpha+5)n^3$ & $[0.562544,1.437456]$ \\
\textsf{Polar Express}~\citep{amsel2026polar} & $n^2(2\alpha+2)$ & $(15\alpha+5)n^3$ & $[0.858872,1.141128]$ \\
\textsf{Gram Newton--Schulz}~\citep{zhang2026gram} & $n^2(2\alpha+4)$ & $(6\alpha+14)n^3$ & $[0.787021,1.212979]$ \\
\midrule
\textbf{\textsf{Polar Switch 2-55}} & $\boldsymbol{n^2(2\alpha+2)}$ & $\boldsymbol{(6\alpha+10)n^3}$ & $[0.872263,1.106236]$ \\
\textsf{Polar Switch 2-64} & $n^2(2\alpha+2)$ & $(6\alpha+10)n^3$ & $[0.897110,1.081213]$ \\
\textsf{Polar Switch 2-54} & $n^2(2\alpha+2)$ & $(6\alpha+9)n^3$ & $[0.771715,1.201363]$ \\
\textsf{Polar Switch 2-512} & $n^2(2\alpha+2)$ & $(9\alpha+8)n^3$ & $[0.894082,1.091006]$ \\
\textsf{Polar Switch 2-4111} & $n^2(2\alpha+2)$ & $(12\alpha+7)n^3$ & $[0.930600,1.056149]$ \\
\textsf{Polar Switch 2-11111} & $n^2(2\alpha+2)$ & $(15\alpha+5)n^3$ & $[0.867287,1.118678]$ \\
\midrule
\textbf{\textsf{Polar Switch 3-55}} & $\boldsymbol{n^2(2\alpha+3)}$ & $\boldsymbol{(6\alpha+10)n^3}$ & $[0.898309,1.080954]$ \\
\textsf{Polar Switch 3-64} & $n^2(2\alpha+3)$ & $(6\alpha+10)n^3$ & $[0.913241,1.070202]$ \\
\textsf{Polar Switch 3-54} & $n^2(2\alpha+3)$ & $(6\alpha+9)n^3$ & $[0.800196,1.178895]$ \\
\bottomrule
\end{tabular}
\end{table}

We make three contributions.
(i) We introduce \texttt{Polar Switch}, a FLOP-efficient orthogonalization method for Muon based on high-degree polynomial circuits.
(ii) We develop efficient GPU implementations of \texttt{Polar Switch} with two or three square buffers.
Each kernel combines a matrix product and its additions before one \texttt{float16} store, and we select coefficients that try to keep this half-precision computation numerically stable.
(iii) Across three LLM pretraining setups, \texttt{Polar Switch} reaches final validation losses similar to those of established orthogonalization methods.

The remainder of the paper is structured as follows. \Cref{sec:related-work} reviews related work.
\Cref{sec:method} presents \texttt{Polar Switch}, explains how we design its circuits and coefficients under a FLOP budget, and describes its GPU implementation.
\Cref{sec:experiments} compares \texttt{Polar Switch} with quintic baselines in LLM
pretraining, runtime, and approximation accuracy.
\Cref{app:coefficients} gives the theory, the numerical search, and the coefficients of \texttt{Polar Switch}.
\Cref{app:campaign} gives the detailed experimental setup for LLM pretraining
and the runtime and approximation accuracy experiments.

\section{Related Work}
\label{sec:related-work}

\paragraph{Polar approximation and polynomial design.}
Polynomial polar iterations~\citep{bjorck1971iterative} and matrix sign methods~\citep{higham1986computing,higham2008functions}
are used in eigenvalue algorithms~\citep{denman1976matrix,nakatsukasa2013stable}, control, and model reduction~\citep{roberts1980linear}.
Rational alternatives include sign iterations~\citep{kenney1991rational}, QR-based dynamically weighted Halley (QDWH)~\citep{nakatsukasa2010optimizing},
and Zolotarev approximants~\citep{nakatsukasa2016computing}. Polar iterations also have backward stability analyses~\citep{nakatsukasa2012backward}.
Broader work covers matrix roots~\citep{higham1987computing,higham1997stable,deadman2013blocked,bini2005algorithms,iannazzo2006newton},
fractional powers~\citep{higham2011schur}, general matrix function computation~\citep{davies2003schur,hale2008computing,guttel2013rational,higham2005scaling,almohy2011computing},
and manifold retractions using projected polynomials~\citep{gawlik2018high}.
Minimax approximation minimizes the largest error over an interval, with solutions characterized by equioscillation~\citep{trefethen2019approximation}
and computed by Remez methods~\citep{pachon2009barycentric}.
Related problems include polynomial sign approximation~\citep{eremenko2007uniform}, rational minimax~\citep{filip2018rational}
and AAA approximation~\citep{nakatsukasa2018aaa}, and Parks--McClellan filter design~\citep{parks1972chebyshev,mcclellan1973computer}.
Paterson--Stockmeyer~\citep{paterson1973number} and related methods~\citep{sastre2018efficient} reduce the number of matrix multiplications.
Other work characterizes which polynomials can be evaluated with a fixed number of products~\citep{jarlebring2026polynomial}.
Chebyshev acceleration~\citep{grishina2026accelerating} constructs odd polynomials with large derivatives at zero to increase small singular values more quickly.
\texttt{Polar Express}~\citep{amsel2026polar} proves that greedy minimax fitting is optimal for a prescribed degree sequence, and its implementation uses quintics.
Dynamic programming has also been used to select minimax compositions for homomorphic comparison under multiplication and depth constraints~\citep{lee2022minimax}.
\texttt{Polar Switch} selects high-degree circuits under FLOP and buffer budgets that account for symmetric GPU kernels and adds a half-precision rounding constraint for numerical stability.

\paragraph{Efficient Muon orthogonalization.}
Equivalent formulas can round differently because their intermediate computations differ~\citep{higham2002accuracy}, and small singular values are sensitive to reduced precision~\citep{boutsikas2024small}.
Mixed precision is studied in linear algebra~\citep{baboulin2009accelerating,carson2017new,abdelfattah2021survey}
and training~\citep{gupta2015deep,wang2018training}.
\texttt{Gram Newton--Schulz}~\citep{zhang2026gram} accumulates square multipliers between rectangular updates, and Dion3~\citep{amsel2026dion3} uses it for distributed training.
AMO~\citep{zhuang2026amo} adapts iteration counts, Dion~\citep{ahn2025dion} uses low-rank updates,
Dion2~\citep{ahn2025dion2} selects rows or columns, Hierarchical Muon~\citep{tang2026hierarchical} uses tiles,
and CacheMuon~\citep{dev2026cachemuon} reuses transforms.
Turbo-Muon~\citep{boissin2026turbo} uses an almost-orthogonal preconditioner for faster updates, and PRISM~\citep{yang2026prism} computes matrix functions adaptively.
IFNSO~\citep{hu2026ifnso} replaces the iterations with one polynomial in selected matrix powers and learns its coefficients.
A cubic Newton--Schulz schedule~\citep{huang2026how} uses ten instead of fifteen dominant matrix products and reaches a similar validation loss.
MuonBP~\citep{khaled2026muonbp}, CMuon~\citep{chen2026cmuon}, and DMuon~\citep{chen2026dmuon} support distributed computation.
\texttt{Polar Switch} builds on polynomial Newton--Schulz
iterations with efficient GPU implementations.

\paragraph{Muon and matrix optimization.}
Muon~\citep{jordan2024muon} forms matrix updates by approximately orthogonalizing matrix momentum
and has been used for efficient LLM training~\citep{liu2025muon,ai2025practical}.
Other optimizers use diagonal gradient statistics, as in AdaGrad~\citep{duchi2011adaptive},
or factored second moments, as in Adafactor~\citep{shazeer2018adafactor}.
K-FAC~\citep{martens2015optimizing} and EKFAC~\citep{george2018fast} approximate curvature using Kronecker factors.
Shampoo~\citep{gupta2018shampoo} and Distributed Shampoo~\citep{shi2023distributed} use matrix preconditioners,
and SOAP~\citep{vyas2025soap} applies Adam in the eigenbasis of Shampoo's preconditioners.
Spectral descent~\citep{carlson2015stochastic,carlson2015preconditioned} exploits singular vectors.
Analyses of Muon cover finite Newton--Schulz iterations~\citep{kim2026convergence}, inexact polar updates~\citep{shulgin2025beyond},
modular duality~\citep{bernstein2025modular}, norm-constrained optimization~\citep{pethick2025training},
preconditioning~\citep{ma2026preconditioning} (including PolarGrad~\citep{lau2026polargrad} and Newton--Muon~\citep{du2026newton}), curvature~\citep{su2025isotropic},
stability~\citep{nguyen2026spectral,qi2026delving}, convergence~\citep{shen2026convergence,li2025note}, and batch size~\citep{sato2026convergence}.
Broader spectral scaling theory addresses feature learning~\citep{yang2024spectral}.
Extensions add symmetry constraints~\citep{lau2026symmetry} or change the normalization~\citep{li2026normuon},
the preconditioning~\citep{xu2026fismo}, the optimization subspace~\citep{refael2025sumo},
or the coordinate system~\citep{gong2026aro}.
Activation-Keyed Momentum (AK-Momentum)~\citep{hong2026activation} introduces anisotropic momentum.
DeMuon~\citep{he2025demuon} addresses decentralized optimization. Applications include diffusion training~\citep{li2026scaling},
low-rank adaptation with ISO-LoRA~\citep{zhu2026rank}, and LLM training~\citep{khona2026soap}.
\texttt{Polar Switch} changes Muon's polar approximation
by designing the polynomial circuits used to form its matrix updates.

\section{Polar Switch}
\label{sec:method}

We develop \texttt{Polar Switch} in three steps.
First, we define polynomial circuits and how they change the singular values (\Cref{sec:circuits}).
Second, we describe how we implement each circuit on a GPU in half-precision floating point and how we try to make it numerically stable (\Cref{sec:implementation}).
Third, we state the design problem that \texttt{Polar Switch} uses to choose its circuits and their coefficients, and we give the resulting \texttt{Polar Switch} methods (\Cref{sec:stage-solve,sec:search}).

\subsection{Polynomial Circuits}
\label{sec:circuits}

Let $X\in\mathbb{R}^{n\times m}$ with $n\le m$.
Each outer iteration $t$ (\Cref{def:circuit}) calculates a polynomial $h_t(XX^\mathsf{T})$ of the Gram matrix and updates $X\leftarrow h_t(XX^\mathsf{T})X$.
Evaluating a high-degree polynomial $h_t$ at $XX^\mathsf{T}$ needs intermediate matrices, and the GPU memory limits how many of them we can store.
\texttt{Polar Switch} therefore stores them in $g\ge2$ reusable square buffers $B_1,\ldots,B_g\in\mathbb{R}^{n\times n}$.

\begin{definition}[Outer iteration and circuit]
\label{def:circuit}
An \emph{outer iteration} updates $X$ with a \emph{circuit}, a fixed sequence of buffer updates that forms a straight-line program~\citep{miller1988efficient}.
The circuit first stores $XX^\mathsf{T}$ in $B_1$ and then performs $k$ \emph{inner iterations} $\tau=1,\ldots,k$ in sequence:
\begin{equation}
 B_{w_\tau}\leftarrow a_\tau B_{i_\tau}B_{j_\tau}+b_\tau B_{r_\tau}+c_\tau I,
 \qquad w_\tau,i_\tau,j_\tau,r_\tau\in\{1,\ldots,g\},\quad w_\tau\notin\{i_\tau,j_\tau\}.
 \label{eq:instruction}
\end{equation}
The output buffer is $B_{w_\tau}$, the two factors are $B_{i_\tau}$ and $B_{j_\tau}$, the addend is $B_{r_\tau}$, and $a_\tau,b_\tau,c_\tau$ are real coefficients.
We call $k$ the \emph{inner depth}.
The outer iteration ends with $X\leftarrow B_{w_k}X$.
A circuit obeys three rules.
First, an inner iteration reads only buffers that hold values, namely $B_1$ and the buffers that earlier inner iterations wrote, so $i_\tau,j_\tau,r_\tau\in\{w_0,w_1,\ldots,w_{\tau-1}\}$ with $w_0=1$.
Second, the output index is not a factor index, $w_\tau\notin\{i_\tau,j_\tau\}$.
Third, the output index can be the addend index, $w_\tau=r_\tau$.
\Cref{sec:implementation} uses the GPU kernel to explain the second and third rules.
\Cref{alg:standard} gives the general form of \texttt{Polar Switch}, and \Cref{alg:polar-switch-methods} gives the circuits and coefficients of each method.
Each outer iteration $t$ has its own circuit, and we omit $t$ from its indices and coefficients.
\end{definition}

\begin{algorithm}[htbp]
\caption{General form of \texttt{Polar Switch}}
\label{alg:standard}
\begin{algorithmic}[1]
\Require $X\in\mathbb{R}^{n\times m}$, $n\le m$; $T$ outer iterations, where outer iteration $t$ has a circuit (\Cref{def:circuit}) with $k_t$ inner iterations
\Ensure Approximation to the polar factor of $X$
\State $X\gets\texttt{float32}(X);\quad X\gets X/(\lVert X\rVert_F+10^{-7});\quad X\gets\texttt{float16}(X)$
\For{$t=1,\ldots,T$}
  \State $B_1\gets XX^\mathsf{T}$
  \For{$\tau=1,\ldots,k_t$}
    \State $B_{w_\tau}\gets a_\tau B_{i_\tau}B_{j_\tau}+b_\tau B_{r_\tau}+c_\tau I$ \Comment{One fused kernel; $w_\tau\notin\{i_\tau,j_\tau\}$}
  \EndFor
  \State $X\gets B_{w_{k_t}}X$
\EndFor
\State \Return $X$
\end{algorithmic}
\end{algorithm}

Each initialized buffer holds a polynomial in $XX^\mathsf{T}$, because $B_1=XX^\mathsf{T}$ and each inner iteration combines buffers with a product, a sum, and a diagonal shift.
Let $X=U\Sigma V^\mathsf{T}$ be the compact SVD of $X$, so that $XX^\mathsf{T}=U\Sigma^2U^\mathsf{T}$.
Then after inner iteration $\tau$, each initialized buffer is $B_i=U\beta_i^{(\tau)}(\Sigma)U^\mathsf{T}$,\footnote{If $X$ has rank below $n$, each buffer also has the term $\beta_i^{(\tau)}(0)(I-UU^\mathsf{T})$. This term vanishes in $B_{w_k}X$, so the outer iteration below is the same for all ranks.} where the polynomials $\beta_i^{(\tau)}$ in the singular value $x$ follow the scalar circuit
\begin{equation}
 \beta_1^{(0)}=x^2,
 \qquad
 \beta_{w_\tau}^{(\tau)}=a_\tau\beta_{i_\tau}^{(\tau-1)}\beta_{j_\tau}^{(\tau-1)}+b_\tau\beta_{r_\tau}^{(\tau-1)}+c_\tau,
 \qquad
 \beta_i^{(\tau)}=\beta_i^{(\tau-1)}\ (i\ne w_\tau),
 \qquad \tau=1,\ldots,k.
 \label{eq:scalar-circuit}
\end{equation}
The last output buffer holds $B_{w_k}=h_t(XX^\mathsf{T})$.
The outer iteration gives $h_t(XX^\mathsf{T})X=Up_t(\Sigma)V^\mathsf{T}$ with the odd polynomial $p_t(x)=x\beta_{w_k}^{(k)}=xh_t(x^2)$ of degree $1+2\deg h_t$.
Thus, it replaces each singular value $x$ with $p_t(x)$, provided $p_t(x)\ge0$.
After $T$ outer iterations with polynomials $p_1,\ldots,p_T$, the result is $Up(\Sigma)V^\mathsf{T}$ with $p=p_T\circ\cdots\circ p_1$.

\begin{definition}[Circuit family]
\label{def:class}
Let $\mathcal C_k^g$ be the set of circuits with inner depth $k$ on $g$ square buffers, over all buffer indices $w_\tau,i_\tau,j_\tau,r_\tau$ and coefficients $a_\tau,b_\tau,c_\tau$, $\tau=1,\ldots,k$.
A circuit $C\in\mathcal C_k^g$ calculates the polynomial $p_C(x)=x\beta_{w_k}^{(k)}$ in \eqref{eq:scalar-circuit}.
The \emph{family}
\[
 \mathcal P_k^g=\left\{p_C:C\in\mathcal C_k^g\right\}
\]
is the set of odd polynomials that $k$ inner iterations on $g$ square buffers can calculate.
\end{definition}

Let $\mathcal O_{2^{k+1}+1}$ be the set of odd polynomials of degree at most $2^{k+1}+1$.
The degree of a product is at most the sum of the degrees of its factors, so $k$ inner iterations give $\deg h_t\le2^k$, and
\begin{equation}
 \mathcal P_k^g\subseteq\mathcal O_{2^{k+1}+1}.
 \label{eq:family-degree}
\end{equation}
A polynomial in $\mathcal O_{2^{k+1}+1}$ has $2^k+1$ coefficients, but for fixed buffer indices, a polynomial in $\mathcal P_k^g$ depends on only the $3k$ coefficients $a_\tau,b_\tau,c_\tau$, $\tau=1,\ldots,k$.\footnote{For $k\ge2$, the number of degrees of freedom is smaller than $3k$.
For $\tau<k$, multiplying the output of inner iteration $\tau$ by a constant $\kappa\ne0$ and dividing its later uses by $\kappa$ does not change $p_C$ (\Cref{app:intervals}).
For example, each nonzero $a_\tau$ with $\tau<k$ can be set to $1$, so each choice of buffer indices gives at most $2k+1$ degrees of freedom.
For $k\ge3$, $2k+1<2^k+1$, so each choice parametrizes a set of measure zero in $\mathcal O_{2^{k+1}+1}$, and the finitely many choices give a proper subset $\mathcal P_k^g$.}
A central step of \texttt{Polar Switch} is to optimize the circuits in $\mathcal C_k^g$.

For example, the first circuit of \textsf{PS3-55} in \eqref{eq:ps355} is in $\mathcal C_5^3$, and each inner iteration squares the output of the previous one ($i_\tau=j_\tau=w_{\tau-1}$, $a_\tau\ne0$).
Thus, $\deg h_t$ doubles to $2$, $4$, $8$, $16$, and $32$, and $p_t$ has degree $65=2^{k+1}+1$, the largest degree in \eqref{eq:family-degree}.
The iterations of Jordan~\citep{jordan2024muon}, \texttt{Polar Express}, and \texttt{Gram Newton--Schulz} can also be written as circuits.

\begin{example}[Standard quintic Newton--Schulz]
\label{ex:quintic}
The quintic iteration \eqref{eq:one-step} of Jordan~\citep{jordan2024muon} and \texttt{Polar Express} adds the linear term $c_tX$ in the final update $X\leftarrow ZX+c_tX$.
Since $(Z+c_tI)X=ZX+c_tX$, it is equivalent in exact arithmetic to a circuit with one inner iteration on $g=2$ buffers:
\[
 B_2\leftarrow a_tB_1^2+b_tB_1+c_tI,
 \qquad
 X\leftarrow B_2X.
\]
In finite precision, the two forms can give different results, because they round different intermediate matrices.
\end{example}

\texttt{Polar Express}~\citep[Remark~3.2]{amsel2026polar} permits different polynomial degrees across iterations, but its implementation uses degree five.
The remark refers to dynamic programming that chooses the best sequence of polynomial iterations under a budget on the number of multiplications~\citep{lee2022minimax}.
If each matrix multiplication has the same cost, a nonic iteration (four multiplications) is 33\% more expensive than a quintic iteration (three).
However, the FLOP cost on GPUs is different from the number of matrix multiplications.
With symmetric kernels, a quintic iteration costs $(3\alpha+1)n^3$ FLOPs and a nonic iteration costs $(3\alpha+2)n^3$ FLOPs, where $\alpha=m/n$ (\Cref{prop:cost}).
A nonic iteration thus needs only 25\% more FLOPs at $\alpha=1$, and less for rectangular matrices.
Polynomial iterations of higher degree therefore need fewer FLOPs on GPUs than the multiplication count suggests.
This lower cost can allow more accurate polar approximations under the same FLOP budget, and \texttt{Polar Switch} uses it to construct circuits of high degree.

\begin{example}[Gram Newton--Schulz]
\label{ex:gns}
The standard \texttt{Gram Newton--Schulz} (\textsf{GNS}) iteration~\citep{zhang2026gram} uses the coefficients $(a_j,b_j,c_j)$, $j=1,\ldots,5$, of \texttt{Polar Express}, with an extra safety factor $1.05$.
Starting from $R_0\leftarrow XX^\mathsf{T}$, the iterations are:
\[
 \begin{array}{@{}llll@{}}
 Z_1\leftarrow b_1R_0+a_1R_0^2, & Q_1\leftarrow Z_1+c_1I, & Y_1\leftarrow R_0Z_1+c_1R_0, & R_1\leftarrow Z_1Y_1+c_1Y_1,\\
 Z_2\leftarrow b_2R_1+a_2R_1^2, & Q_2\leftarrow Q_1Z_2+c_2Q_1, & X\leftarrow Q_2X, & R_2\leftarrow XX^\mathsf{T},\\
 Z_3\leftarrow b_3R_2+a_3R_2^2, & Q_3\leftarrow Z_3+c_3I, & Y_3\leftarrow R_2Z_3+c_3R_2, & R_3\leftarrow Z_3Y_3+c_3Y_3,\\
 Z_4\leftarrow b_4R_3+a_4R_3^2, & Q_4\leftarrow Q_3Z_4+c_4Q_3, & Y_4\leftarrow R_3Z_4+c_4R_3, & R_4\leftarrow Z_4Y_4+c_4Y_4,\\
 Z_5\leftarrow b_5R_4+a_5R_4^2, & Q_5\leftarrow Q_4Z_5+c_5Q_4, & X\leftarrow Q_5X.
 \end{array}
\]
We use the names of the official code, except for $Y_j$.
$R_j$ tracks the Gram matrix after quintic step $j$, $Q_j$ collects the matrix that multiplies $X$, and $Z_j$ and $Y_j$ are intermediate matrices.
In exact arithmetic, it has the equivalent circuit form
\[
 \begin{array}{@{}llll@{}}
 B_1\leftarrow XX^\mathsf{T} & B_2\leftarrow a_1B_1^2+b_1B_1+c_1I & B_3\leftarrow B_1B_2,\ B_1\leftarrow B_2B_3 & \\
     & B_3\leftarrow a_2B_1^2+b_2B_1+c_2I & B_1\leftarrow B_2B_3 & X\leftarrow B_1X\\
 B_1\leftarrow XX^\mathsf{T} & B_2\leftarrow a_3B_1^2+b_3B_1+c_3I & B_3\leftarrow B_1B_2,\ B_1\leftarrow B_2B_3 & \\
     & B_3\leftarrow a_4B_1^2+b_4B_1+c_4I & B_4\leftarrow B_2B_3,\ B_2\leftarrow B_1B_3,\ B_1\leftarrow B_3B_2 & \\
     & B_3\leftarrow a_5B_1^2+b_5B_1+c_5I & B_2\leftarrow B_4B_3 & X\leftarrow B_2X.
 \end{array}
\]
Each row runs from left to right, and a product without coefficients has $a_\tau=1$ and $b_\tau=c_\tau=0$.
\end{example}

The two outer iterations of \textsf{GNS} have inner depths $5$ and $9$ on $g=4$ buffers, so they are circuits in $\mathcal C_5^4$ and $\mathcal C_9^4$.
Their polynomials have degrees $25$ and $125$, but $5$ and $9$ inner iterations can reach the full degrees $65$ and $1025$ in \eqref{eq:family-degree}.
The gap comes from the mixed products ($i_\tau\ne j_\tau$), such as $B_3\leftarrow B_1B_2$ and $B_1\leftarrow B_2B_3$, which track the Gram matrix of the updated $X$.

\texttt{Polar Switch} does not impose full degree as a constraint when optimizing its circuits and their coefficients, but all circuits we found reach full degree. For example, when $g=2$, the alternating writes of our search (\Cref{sec:search}) make each inner iteration square the output of the previous one ($i_\tau=j_\tau=w_{\tau-1}$), and for $g=3$, all circuits that we found also do so, so they reach the full degree.
For example, \textsf{PS3-55} has inner depths $5$ and $5$ and degrees $65$ and $65$ (\Cref{alg:polar-switch-methods}).
This suggests the hypothesis that, for a given inner depth, the most accurate way to approximate the polar factor is to use the highest possible degree, even within the restricted family $\mathcal P_k^g$.
To reach the full degree $2^{k+1}+1$, the largest degree among the buffers must double at every inner iteration.
Then only the output of the previous inner iteration has the largest degree, so each inner iteration must square it, and a circuit with a mixed product has a lower degree.
This explains why \textsf{GNS} reaches a lower total degree, $25\cdot125=3125$, than \textsf{PS2-55}, $65\cdot65=4225$, at a higher FLOP cost (\Cref{tab:circuit-cost}).

\subsection{GPU Implementation, Cost, and Rounding}
\label{sec:implementation}

In \Cref{alg:standard}, each inner iteration $\tau$ runs as one GPU kernel.
Each buffer holds a polynomial in $XX^\mathsf{T}$ (\Cref{sec:circuits}), and polynomials in the same symmetric matrix are symmetric and commute, so each product $B_{i_\tau}B_{j_\tau}=B_{j_\tau}B_{i_\tau}$ is symmetric.
A symmetric matrix multiplication kernel calculates only one triangle of such a product and copies it to the other triangle.
\texttt{Polar Switch} extends such a kernel from the QuACK kernel library (\Cref{app:campaign}).
For inner iteration $\tau$, the kernel multiplies the product by $a_\tau$, adds $b_\tau B_{r_\tau}$, and adds $c_\tau$ to each diagonal entry in its \emph{epilogue}, without allocating an identity matrix.\footnote{In the epilogue of a matrix multiplication kernel,
additional operations, such as scaling and matrix addition, can be applied to the
accumulated product before storing the result in global memory, without requiring
a separate kernel launch.}
The kernel therefore writes each output $B_{w_\tau}$ to memory once, and the additions need no separate kernels or extra matrices.

This kernel explains the second and third rules of \Cref{def:circuit}.
Each entry of the product $B_{i_\tau}B_{j_\tau}$ needs a full row and a full column of the two factors, so the kernel still reads the factors while it writes the output, and the output cannot overwrite a factor ($w_\tau\notin\{i_\tau,j_\tau\}$).
Each entry of the output needs only the same entry of the addend $B_{r_\tau}$, so the output can overwrite the addend ($w_\tau=r_\tau$).

The circuits of \texttt{Polar Switch} use $g$ square buffers $B_1,\ldots,B_g$, each containing $n^2$ elements.
They also use two rectangular buffers,
each containing $mn$ elements. One holds the current matrix $X$.
The other holds the result of the update $X\leftarrow B_{w_k}X$ in \Cref{alg:standard}.
The result needs a separate buffer because the multiplication still needs the entries
of the current $X$. Once the update is complete, the result becomes the new $X$.

The products of a circuit have different costs on a GPU, so we count FLOPs.
We count only the leading-order FLOPs of matrix products, because they are dominant for large matrices.
Let $\alpha=m/n\ge1$ be the aspect ratio.
A symmetric kernel halves the FLOPs of a product because it calculates only one triangle.

\begin{proposition}[Cost and memory of an outer iteration]
\label{prop:cost}
An outer iteration with inner depth $k$ costs
\[
  \underbrace{\alpha n^3}_{XX^\mathsf{T}}
  +\underbrace{kn^3}_{k\text{ products }B_{i_\tau}B_{j_\tau}}
  +\underbrace{2\alpha n^3}_{B_{w_k}X}
  =(3\alpha+k)n^3\ \text{FLOPs}.
\]
With $g$ square buffers and two rectangular buffers, a method stores $n^2(2\alpha+g)$ elements.
\end{proposition}
\begin{proof}
The outer iteration calculates three types of products.
The Gram matrix $XX^\mathsf{T}$ is symmetric, and it costs $mn^2=\alpha n^3$ FLOPs.
Each of the $k$ inner iterations calculates a symmetric $n\times n$ product $B_{i_\tau}B_{j_\tau}$, which costs $n^3$ FLOPs.
The final update $B_{w_k}X$ is a general product, and it costs $2mn^2=2\alpha n^3$ FLOPs.
The $g$ square buffers hold $gn^2$ elements, and the two rectangular buffers hold $2mn=2\alpha n^2$ elements.
\end{proof}

We now explain how \texttt{Polar Switch} tries to keep its half-precision computation numerically stable.
\texttt{Polar Switch} normalizes $X$ in \texttt{float32} and converts it to \texttt{float16}.\footnote{Following the standard \texttt{Gram Newton--Schulz} implementation~\citep{zhang2026gram}, we normalize in \texttt{float32} before converting to \texttt{float16}. We use \texttt{float16} rather than \texttt{bfloat16} because it has 10 fraction bits instead of 7.}
During the iterations, $X$ and all buffers are stored in \texttt{float16}.
Each product runs on \texttt{float16} Tensor Cores and accumulates in \texttt{float32}.
The epilogue applies the scaling by $a_\tau$, the addend $b_\tau B_{r_\tau}$, and the diagonal shift $c_\tau I$ in \texttt{float32}, and then the kernel converts the result to \texttt{float16} and stores it in $B_{w_\tau}$.
This avoids an additional square buffer and an intermediate \texttt{float16} matrix, which reduces memory traffic and avoids an additional rounding step.
The main rounding of an inner iteration is likely the \texttt{float16} store of its result, which rounds each entry to the nearest \texttt{float16} value.
The Gram matrix $B_1\leftarrow XX^\mathsf{T}$ and the output $X\leftarrow B_{w_k}X$ are stored in the same way, so an outer iteration with inner depth $k$ has $k+2$ stores.
To try to select coefficients $a_\tau,b_\tau,c_\tau$ for which the circuit stays stable under these roundings, we add a multiplicative factor for each store to the scalar circuit \eqref{eq:scalar-circuit}.

\begin{definition}[Perturbed outer iteration]
\label{def:perturbed}
Let $C\in\mathcal C_k^g$ be a circuit with buffer indices $w_\tau,i_\tau,j_\tau,r_\tau$ and coefficients $a_\tau,b_\tau,c_\tau$, $\tau=1,\ldots,k$, let $s\ge1$ be a safety factor, and let $\mathbf{v}=(v_0,\ldots,v_{k+1})\in[s^{-1},s]^{k+2}$.
The perturbed outer iteration is
\begin{equation}
\begin{aligned}
 &\widetilde\beta_1^{(0)}=v_0x^2,
 \qquad
 \widetilde\beta_{w_\tau}^{(\tau)}=v_\tau\left(a_\tau\widetilde\beta_{i_\tau}^{(\tau-1)}\widetilde\beta_{j_\tau}^{(\tau-1)}+b_\tau\widetilde\beta_{r_\tau}^{(\tau-1)}+c_\tau\right),\\
 &\widetilde\beta_i^{(\tau)}=\widetilde\beta_i^{(\tau-1)}\ (i\ne w_\tau),
 \qquad \tau=1,\ldots,k,
 \qquad
 \widetilde p_C(x;\mathbf{v})=v_{k+1}\,x\,\widetilde\beta_{w_k}^{(k)}.
\end{aligned}
 \label{eq:perturbed-circuit}
\end{equation}
The rounding factors $v_0$, $v_\tau$, and $v_{k+1}$ model the stores of the Gram matrix, of inner iteration $\tau$, and of the output.
For $\mathbf{v}=(1,\ldots,1)$, $\widetilde p_C(x;\mathbf{v})=p_C(x)$.
\end{definition}

\subsection{Design Problem}
\label{sec:stage-solve}

In exact arithmetic, after the normalization in \Cref{alg:standard}, each singular value of $X$ is at most one.
As in \texttt{Polar Express}, we optimize the polynomials for singular values in $[\ell_0,1]$ with $\ell_0=0.001$ (\Cref{app:coefficients} explains this choice).
We fix the number $g$ of square buffers.
A \emph{schedule} $(k_1,\ldots,k_T)$ lists the inner depths of $T$ outer iterations, and its cost is $\sum_{t=1}^T(3\alpha+k_t)n^3$ FLOPs (\Cref{prop:cost}).
We want to find the schedule and the polynomials $p_t\in\mathcal P_{k_t}^g$ with the smallest worst-case error under a FLOP budget $f_{\max}n^3$.
In matrix form, the problem is
\begin{equation}
  \min_{\substack{T\ge1,\ k_1,\ldots,k_T\ge1\\ \sum_{t=1}^T(3\alpha+k_t)\le f_{\max}}}
  \ \inf_{p_t\in\mathcal P_{k_t}^g,\ t=1,\ldots,T}
  \ \max_{\substack{X\in\mathbb{R}^{n\times m}\\
    \sigma_i(X)\in[\ell_0,1],\ i=1,\ldots,n}}
  \left\|\operatorname{polar}(X)-p(X)\right\|_2,
  \label{eq:matrix-budget-problem}
\end{equation}
where $\sigma_1(X),\ldots,\sigma_n(X)$ are the singular values of $X$, $p=p_T\circ\cdots\circ p_1$ is the composition that the circuits calculate, and $p(X)=Up(\Sigma)V^\mathsf{T}$ for the compact SVD $X=U\Sigma V^\mathsf{T}$ is the output of the $T$ outer iterations in exact arithmetic.

\begin{proposition}[Scalar reduction]
\label{prop:scalar}
The matrix problem \eqref{eq:matrix-budget-problem} has the same objective, and hence the same value and the same minimizers when they exist, as the scalar minimax problem
\begin{equation}
  \min_{\substack{T\ge1,\ k_1,\ldots,k_T\ge1\\ \sum_{t=1}^T(3\alpha+k_t)\le f_{\max}}}
  \ \inf_{p_t\in\mathcal P_{k_t}^g,\ t=1,\ldots,T}
  \ \max_{x\in[\ell_0,1]}\left|1-p(x)\right|.
  \label{eq:composition-minimax}
\end{equation}
\end{proposition}
\begin{proof}
Since $\operatorname{polar}(X)=UV^\mathsf{T}$ and $U$ and $V$ have orthonormal columns, the spectral norm satisfies
\[
  \|{\operatorname{polar}(X)-p(X)}\|_2
  =\|I-p(\Sigma)\|_2
  =\max_{1\le i\le n}
  \left|1-p(\sigma_i(X))\right|
  \le\max_{x\in[\ell_0,1]}\left|1-p(x)\right|.
\]
The bound is attained when at least one singular value of $X$ is a maximizer of $|1-p(x)|$ on $[\ell_0,1]$, and such an $X$ is admissible in \eqref{eq:matrix-budget-problem}.
Thus, the two problems have the same objective for each schedule and each choice of $p_t\in\mathcal P_{k_t}^g$, $t=1,\ldots,T$.
\end{proof}

We design the approximation in \eqref{eq:composition-minimax} through intervals $[\ell_t,u_t]$ with $0\le\ell_t\le u_t$, $u_t>0$, and $(p_t\circ\cdots\circ p_1)([\ell_0,1])\subseteq[\ell_t,u_t]$ in exact arithmetic, $t=1,\ldots,T$.
Problem \eqref{eq:composition-minimax} asks for a final interval $[\ell_T,u_T]$ that is as close to one as possible.
Only the ratios $\ell_t/u_t$ matter, because an outer iteration can scale its output by any $\lambda>0$, and the next outer iteration can undo this scaling, without additional matrix products.
First, multiplying the coefficients $a_{k_t},b_{k_t},c_{k_t}$ of the last inner iteration of outer iteration $t$ by $\lambda$ multiplies $B_{w_{k_t}}$ and the output $B_{w_{k_t}}X$ by $\lambda$, so $[\ell_t,u_t]$ becomes $[\lambda\ell_t,\lambda u_t]$.
Second, outer iteration $t+1$ then receives $\lambda X$ in place of $X$, so its Gram matrix $B_1$ is multiplied by $\lambda^2$.
While $B_1$ holds this Gram matrix, we divide $a_\tau$ by $\lambda^2$ for each factor $B_1$ and $b_\tau$ by $\lambda^2$ if the addend is $B_1$.
For example, in the first inner iteration $B_{w_1}\leftarrow a_1B_1^2+b_1B_1+c_1I$, we divide $a_1$ by $\lambda^4$ and $b_1$ by $\lambda^2$.
Then each inner iteration of outer iteration $t+1$ stores the same matrix as before.
Third, we divide $a_{k_{t+1}},b_{k_{t+1}},c_{k_{t+1}}$ by $\lambda$, so $B_{w_{k_{t+1}}}$ is divided by $\lambda$, and the output $B_{w_{k_{t+1}}}(\lambda X)$ is the same as before.
Thus, all later outer iterations give the same outputs as before.
For $t=T$, no outer iteration follows, and the final centering (\Cref{thm:relation}) chooses the scale.
Therefore, we fix $u_t=1$ for all $t$ and maximize the lower endpoint $\ell_T$.

In the rounding model \eqref{eq:perturbed-circuit}, we also require each perturbed outer iteration to map $[0,1]$ into $[0,1]$, so that each output is a valid input of the next one.
We choose each $p_t\in\mathcal P_{k_t}^g$ through a circuit $C_t\in\mathcal C_{k_t}^g$ with $p_t=p_{C_t}$.
For $g$ square buffers, a FLOP budget $f_{\max}n^3$, and a largest inner depth $k_{\max}$, this gives the \emph{\texttt{Polar Switch} design problem}
\begin{equation}
\begin{aligned}
 &\max_{\substack{T\ge1,\ 1\le k_1,\ldots,k_T\le k_{\max}\\ \sum_{t=1}^T(3\alpha+k_t)\le f_{\max}}}\
 \max_{\substack{C_t\in\mathcal C_{k_t}^g,\ \ell_t\in[0,1]\\ t=1,\ldots,T}}\
 \ell_T\\[4pt]
 &\qquad\text{s.t.}\quad
 p_{C_t}([\ell_{t-1},1])\subseteq[\ell_t,1],\quad
 0\le\widetilde p_{C_t}(x;\mathbf{v})\le1,\quad x\in[0,1],\ \mathbf{v}\in[s^{-1},s]^{k_t+2},\ t=1,\ldots,T.
\end{aligned}
\label{eq:full-problem}
\end{equation}
The second constraint keeps each perturbed output in $[0,1]$ for all inputs in $[0,1]$, which is a design rule for numerical stability.
After centering, a feasible point of \eqref{eq:full-problem} has worst-case error at most $(1-\ell_T)/(1+\ell_T)$ in \eqref{eq:matrix-budget-problem} (\Cref{thm:relation}).

\subsection{Solving the Design Problem}
\label{sec:search}

For a given number $g$ of square buffers, we must find the schedule $(k_1,\ldots,k_T)$, and, for each $k_t$, a circuit $C_t\in\mathcal C_{k_t}^g$, which consists of its buffer indices $w_\tau,i_\tau,j_\tau,r_\tau$ and its coefficients $a_\tau,b_\tau,c_\tau$.
\Cref{app:coefficients} treats \eqref{eq:full-problem} in four steps.
First, for a fixed schedule, greedy circuits, which make each $\ell_t$ as large as possible, are optimal (\Cref{thm:greedy}), as for the quintic compositions of \texttt{Polar Express}.
Second, if the per-iteration problems are solved exactly, dynamic programming over costs selects the schedule with the largest $\ell_T$ (\Cref{thm:budget}).
Only this step depends on the aspect ratio $\alpha$, through the cost $(3\alpha+k_t)n^3$ of each outer iteration; the per-iteration problems and their circuits do not.
Third, for each inner depth $k_t$, we find the circuit numerically with a solver that GPT-6 Astra helped design, and interval arithmetic certifies the result (\Cref{app:minimax-solve}).
Fourth, a final scaling centers the interval $[\ell_T,1]$ at one (\Cref{thm:relation}), and a rescaling of the buffers keeps their values in range for \texttt{float16} (\Cref{app:intervals}).

For $g=2$, our numerical search uses alternating writes.
Call the two buffers $A$ and $B$.
An inner iteration that writes into $B$ cannot read $B$ as a factor, so both factors are $A$, and it has the form $B\leftarrow a_\tau A^2+b_\tau A+c_\tau I$ or $B\leftarrow a_\tau A^2+b_\tau B+c_\tau I$.
The next inner iteration in this search writes into $A$, and it has the form $A\leftarrow a_{\tau+1}B^2+b_{\tau+1}B+c_{\tau+1}I$ or $A\leftarrow a_{\tau+1}B^2+b_{\tau+1}A+c_{\tau+1}I$.
With this alternating write order fixed, only the addend of each inner iteration, $A$ or $B$, remains to be chosen.
With $k=6$ inner iterations, this gives $2^5=32$ alternating buffer patterns, because the first addend must be $A$, and we enumerate all of these patterns.

For $g=3$, the factors and the addends have many more choices.
Thus, we optimize only the fixed schedules $(5,5)$, $(6,4)$, and $(5,4)$, and we reduce the search space with degree-bound pruning designed with the help of GPT-6 Astra (\Cref{app:minimax-solve}).
The local search may still miss the optimal solution.

Our solutions are not guaranteed to be globally optimal.
The problem is nonconvex, so a local search can stop at a local optimum.
Also, the rounding constraint must hold for all inputs and all rounding factors.
The local search checks it only on samples, and we certify every accepted circuit by interval arithmetic, which becomes expensive as $k$ grows.

\label{sec:design}
We use $s=1.005$.
The name \textsf{Polar Switch $g$-$k_1\ldots k_T$}, or \textsf{PS$g$-$k_1\ldots k_T$}, gives the number $g$ of square buffers and the schedule $(k_1,\ldots,k_T)$, with each inner depth written as one digit.
\textsf{PS2} and \textsf{PS3} denote the methods with $g=2$ and $g=3$.
For example, \textsf{PS2-55} uses $g=2$ and $k_1=k_2=5$.
For $g=2$ and each number of outer iterations $T=2,3,4,5$, we use the schedule with $k_{\max}=5$ and the smallest total inner depth $k_1+\cdots+k_T$ whose centered final interval, as in \Cref{tab:circuit-cost}, has a lower endpoint of at least $0.85$.
For a fixed $T$, the cost $3\alpha T+k_1+\cdots+k_T$ orders the schedules in the same way for every $\alpha$, so this choice does not depend on $\alpha$.
A variant of the dynamic program of \Cref{app:search} that keeps one schedule for each cost and each $T$ gives the schedules of \textsf{PS2-55}, \textsf{PS2-512}, \textsf{PS2-4111}, and \textsf{PS2-11111} for $g=2$.
We also include \textsf{PS2-54} and \textsf{PS2-64} with fixed schedules, for which we optimize only the circuits.
\textsf{PS2-54} costs $n^3$ fewer FLOPs than \textsf{PS2-55} and does not reach $0.85$, and \textsf{PS2-64} has a better final interval than \textsf{PS2-55} at the same cost.
For $g=3$, \textsf{PS3-55}, \textsf{PS3-64}, and \textsf{PS3-54} use the fixed schedules above.
\Cref{tab:circuit-cost} gives the memory, FLOPs, and final intervals of all methods, and \Cref{alg:polar-switch-methods} gives all \texttt{Polar Switch} methods.

\section{Experiments}
\label{sec:experiments}

Each experiment run uses a single NVIDIA B200 GPU.
In \Cref{sec:training}, we compare \texttt{Polar Switch} with the baselines in LLM training by final validation loss and numerical stability.
In \Cref{sec:synthetic}, we compare the runtime and the approximation accuracy.
The four baselines are \textsf{Jordan}, \textsf{PE}, and \textsf{PE (BF16)} (\Cref{alg:baselines}), and \textsf{GNS} (\Cref{app:campaign}).
\textsf{PE (BF16)} uses the coefficients of \textsf{PE} with \texttt{bfloat16} storage and products and the normalization of \Cref{alg:baselines}.
\textsf{Jordan} and \textsf{GNS} are centered at the last iteration. They multiply the output of their last iteration by a constant, so that their final interval is symmetric about one.\footnote{\label{fn:jordan-centering}Without centering, the 5 iterations of \textsf{Jordan} map $[0.001,1]$ into $[0.470543,1.202369]$, which is not symmetric about one. Thus, its updates have a smaller norm than those of the other methods. In our earlier Gemma 3-1B pretraining experiments, we used \textsf{Jordan} without centering, and its validation loss was higher than other methods. After we centered its final interval at one, the validation losses became similar.}
\Cref{app:campaign} gives the details of the experimental setup.

\subsection{LLM Experiments}
\label{sec:training}

We use three LLM pretraining setups, which are Modded-NanoGPT \#91~\citep{jordan2024moddednanogpt}, Qwen3-0.6B~\citep{qwen2025qwen3}, and Gemma 3-1B~\citep{gemma2025gemma3}.
In each setup, all methods use the same training settings, and only the polar approximation for the matrix updates of Muon (NorMuon in Modded-NanoGPT \#91) changes.
\Cref{tab:training-setup} summarizes the setups, and \Cref{app:campaign}
gives the implementation details.

\begin{table}[htbp]
\centering
\caption{LLM training setups. Of the \texttt{Polar Switch} methods, the Gemma 3-1B setup uses only the methods with two outer iterations. Total FLOPs are estimated for one run and depend on the orthogonalization method.}
\label{tab:training-setup}
\begin{tabular}{@{}lrrrrr@{}}
\toprule
Model & Parameters & Tokens & Total FLOPs & Seeds & Total runs \\
\midrule
Modded-NanoGPT \#91 & 637,656,487 & 343,408,640 & $(2.413\text{--}2.418)\times10^{17}$ & 20 & 260 \\
Qwen3-0.6B & 596,049,920 & 12,000,428,032 & $(5.148\text{--}5.155)\times10^{19}$ & 1 & 13 \\
Gemma 3-1B & 999,885,952 & 22,000,173,056 & $(1.360\text{--}1.363)\times10^{20}$ & 1 & 10 \\
\bottomrule
\end{tabular}
\end{table}

Modded-NanoGPT \#91 uses the training script of short-track record \#91 of the Modded-NanoGPT speedrun benchmark~\citep{jordan2024moddednanogpt}.
\Cref{fig:training} shows the final validation losses of Modded-NanoGPT \#91, with the same 20 seeds for each method (left), and of Qwen3-0.6B and Gemma 3-1B (right).
\texttt{Polar Switch} reaches final validation losses similar to those of the baselines.

We also measure the numerical stability of all methods by how much the spectral norm of each output exceeds its designed bound.
At each training step, the polar approximation receives 104 matrices in Modded-NanoGPT \#91, 196 in Qwen3-0.6B (7 in each of its 28 layers), and 182 in Gemma 3-1B (7 in each of its 26 layers).
For each matrix that the polar approximation returns, we estimate its spectral norm with Lanczos iteration and compare it with the upper endpoint of the method's final interval in \Cref{tab:circuit-cost}.
Modded-NanoGPT \#91 trains 20 seeds for 1,290 steps, Qwen3-0.6B trains for 22,889 steps, and Gemma 3-1B trains for 41,962 steps.
Thus, each method returns 2,683,200, 4,486,244, and 7,637,084 matrices in the three setups.
\Cref{tab:stability} counts the matrices whose estimate exceeds the bound by 3--30\%, 30--300\%, and more than 300\%.
Only \textsf{GNS} and the circuits with inner depth 6 have such matrices.
\begin{figure}[htbp]
\centering
\begin{minipage}[c]{0.595\linewidth}\centering
\includegraphics{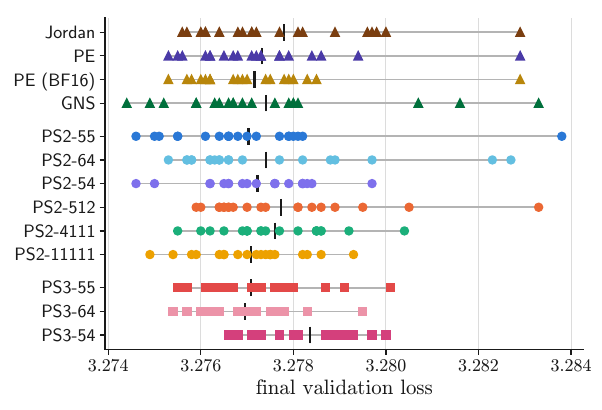}
\end{minipage}\hfill
\begin{minipage}[c]{0.403\linewidth}\centering
\begin{tabular}{@{}lrr@{}}
\toprule
Method & Qwen3-0.6B & Gemma 3-1B \\
\midrule
\textsf{Jordan} & 2.6252 & 2.5534 \\
\textsf{PE} & 2.6255 & 2.5523 \\
\textsf{PE (BF16)} & 2.6260 & 2.5523 \\
\textsf{GNS} & 2.6250 & 2.5515 \\
\midrule
\textsf{PS2-55} & 2.6257 & 2.5520 \\
\textsf{PS2-64} & 2.6258 & 2.5524 \\
\textsf{PS2-54} & 2.6257 & 2.5525 \\
\textsf{PS2-512} & 2.6258 & -- \\
\textsf{PS2-4111} & 2.6257 & -- \\
\textsf{PS2-11111} & 2.6260 & -- \\
\midrule
\textsf{PS3-55} & 2.6255 & 2.5524 \\
\textsf{PS3-64} & 2.6257 & 2.5522 \\
\textsf{PS3-54} & 2.6266 & 2.5533 \\
\bottomrule
\end{tabular}
\end{minipage}
\caption{Final validation loss of Modded-NanoGPT \#91 (left), and of Qwen3-0.6B and Gemma 3-1B (right). On the left, markers show runs with the same 20 seeds, vertical ticks show means, and bars span the minimum and maximum. Triangles mark baselines, circles mark \textsf{PS2}, and squares mark \textsf{PS3}. On the right, all methods use the same seed, and a dash marks methods not trained on Gemma 3-1B.}
\label{fig:training}
\end{figure}

\begin{table}[htbp]
\centering
\caption{Number of matrices whose spectral norm estimate exceeds the bound in column Bound, the upper endpoint in \Cref{tab:circuit-cost} rounded up to four decimals, by 3--30\%, 30--300\%, and more than 300\%, out of 2,683,200 for Modded-NanoGPT \#91, 4,486,244 for Qwen3-0.6B, and 7,637,084 for Gemma 3-1B per method. \textsf{PE (BF16)} uses the bound of \textsf{PE}. A dash marks methods not trained on Gemma 3-1B.}
\label{tab:stability}
\begin{tabular}{@{}lrrrrrrrrrr@{}}
\toprule
 & & \multicolumn{3}{c}{Modded-NanoGPT \#91} & \multicolumn{3}{c}{Qwen3-0.6B} & \multicolumn{3}{c}{Gemma 3-1B} \\
\cmidrule(lr){3-5}\cmidrule(lr){6-8}\cmidrule(lr){9-11}
Method & Bound & 3--30 & 30--300 & $>$300 & 3--30 & 30--300 & $>$300 & 3--30 & 30--300 & $>$300 \\
\midrule
\textsf{Jordan} & 1.4375 & 0 & 0 & 0 & 0 & 0 & 0 & 0 & 0 & 0 \\
\textsf{PE} & 1.1412 & 0 & 0 & 0 & 0 & 0 & 0 & 0 & 0 & 0 \\
\textsf{PE (BF16)} & 1.1412 & 0 & 0 & 0 & 0 & 0 & 0 & 0 & 0 & 0 \\
\textsf{GNS} & 1.2130 & 4 & 0 & 0 & 0 & 0 & 0 & 0 & 0 & 0 \\
\midrule
\textsf{PS2-55} & 1.1063 & 0 & 0 & 0 & 0 & 0 & 0 & 0 & 0 & 0 \\
\textsf{PS2-64} & 1.0813 & 6 & 3 & 0 & 0 & 0 & 0 & 0 & 0 & 0 \\
\textsf{PS2-54} & 1.2014 & 0 & 0 & 0 & 0 & 0 & 0 & 0 & 0 & 0 \\
\textsf{PS2-512} & 1.0911 & 0 & 0 & 0 & 0 & 0 & 0 & -- & -- & -- \\
\textsf{PS2-4111} & 1.0562 & 0 & 0 & 0 & 0 & 0 & 0 & -- & -- & -- \\
\textsf{PS2-11111} & 1.1187 & 0 & 0 & 0 & 0 & 0 & 0 & -- & -- & -- \\
\midrule
\textsf{PS3-55} & 1.0810 & 0 & 0 & 0 & 0 & 0 & 0 & 0 & 0 & 0 \\
\textsf{PS3-64} & 1.0703 & 33 & 5 & 0 & 0 & 0 & 0 & 2 & 0 & 0 \\
\textsf{PS3-54} & 1.1789 & 0 & 0 & 0 & 0 & 0 & 0 & 0 & 0 & 0 \\
\bottomrule
\end{tabular}
\end{table}

In Modded-NanoGPT \#91 training, \textsf{PS2-64} and \textsf{PS3-64} exceeded their bounds by more than 3\% on 9 and 38 of 2,683,200 matrices, and \textsf{PS3-64} also on 2 of 7,637,084 matrices in Gemma 3-1B training (\Cref{tab:stability}), so we recommend inner depths of at most $5$.
A possible reason is that the rounding model \eqref{eq:perturbed-circuit} does not include two effects of the \texttt{float16} stores.
First, the store errors are additive.
Let $E=\texttt{float16}(A)-A$ be the error of storing a matrix $A$ in \texttt{float16}.
Without overflow or underflow, each entry of $E$ is at most $2^{-11}$ (the unit roundoff of \texttt{float16}) times the magnitude of the same entry of $A$, so $\lVert E\rVert_2\le2^{-11}\sqrt n\,\lVert A\rVert_2$.
For $n=2048$ and $\lVert A\rVert_2=1$, this bound is about $0.022$, which is larger than the smallest singular value $0.001$ that the circuits must amplify.
With this additive error, the rounding constraint in \eqref{eq:full-problem} would be nearly impossible to satisfy, so the rounding model \eqref{eq:perturbed-circuit} uses multiplicative factors.

Second, $E$ is in general not a polynomial in $XX^\mathsf{T}$, so the buffers can stop commuting, the singular values can drift, and the singular directions can mix (\Cref{app:drift}).
These errors can grow with the inner depth.
\texttt{Gram Newton--Schulz} also recomputes its Gram matrix from $X$ after two iterations for stability~\citep{zhang2026gram}.
We recommend \textsf{PS2-55} and \textsf{PS3-55}, the bold rows of \Cref{tab:circuit-cost}.

\subsection{Runtime and Approximation Accuracy}
\label{sec:synthetic}

\Cref{fig:campaign-timing-heatmap} shows the runtimes on synthetic single matrices and batches, and the runtime ratios relative to \textsf{PE} and \textsf{GNS}.
For \textsf{PS2-55} and \textsf{PS3-55}, the median runtime ratios are 0.813--0.845 relative to \textsf{PE} on the square shapes and 0.838--0.914 relative to \textsf{GNS} on the rectangular shapes, that is, 15.5--18.7\% and 8.6--16.2\% lower runtime.\footnote{The gain over \textsf{GNS} is larger than the FLOP counts in \Cref{tab:circuit-cost} predict.
One possible reason is that \textsf{GNS} stores the polynomial that multiplies $X$ as a square matrix, so its steps $Q_1\leftarrow Z_1+c_1I$ and $Q_3\leftarrow Z_3+c_3I$ (\Cref{ex:gns}) need an identity matrix.
The official code builds this identity as a separate matrix and adds it in a separate step, while \texttt{Polar Switch} adds $c_\tau$ to the diagonal inside its kernel (\Cref{sec:implementation}).}

\begin{figure}[htbp]
  \centering
  \includegraphics{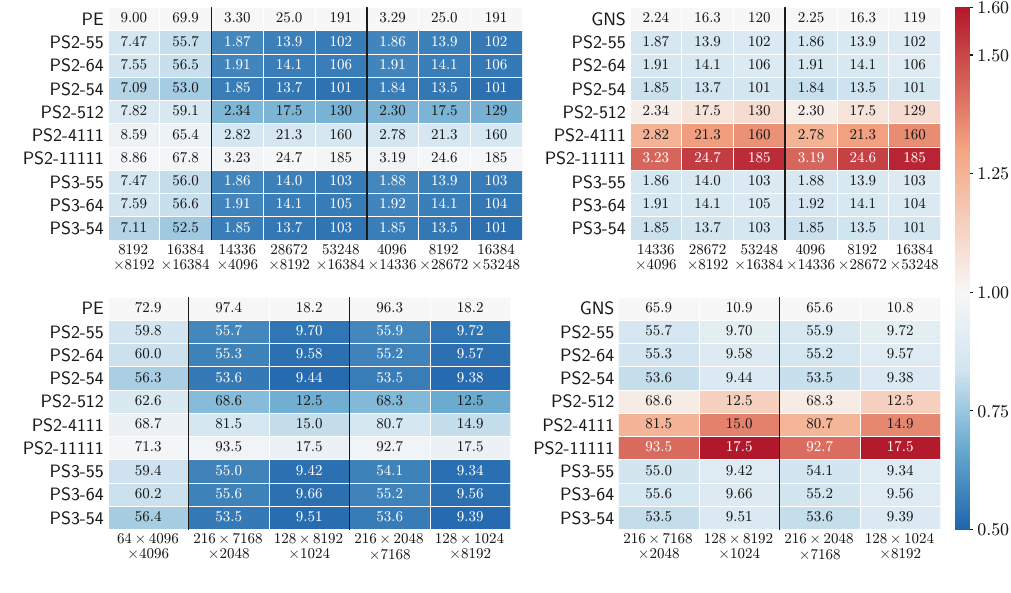}
  \caption{Median runtime in milliseconds for single matrices (top) and batches (bottom).
  Colors show median runtime ratios relative to \textsf{PE} (left) and \textsf{GNS} (right, rectangular shapes only).
  Values below one indicate lower runtime.
  Column labels give tensor dimensions, with the batch size first.}
  \label{fig:campaign-timing-heatmap}
\end{figure}

\begin{figure}[htbp]
  \centering
  \includegraphics{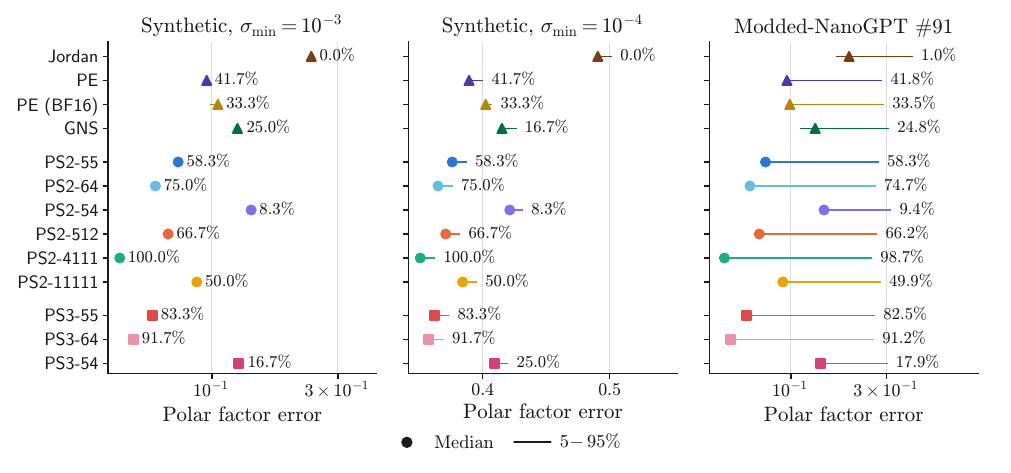}
\caption{Polar factor error \eqref{eq:experimental-polar-error} on synthetic matrices with minimum singular value $\sigma_{\min}=10^{-3}$ (left) and $\sigma_{\min}=10^{-4}$ (middle), 1,500 matrices each, and on 33,592 matrices from one Modded-NanoGPT \#91 run with \textsf{PE} (right).
  Percentages give the fraction of pairs of a matrix and another method of the panel in which the method has the smaller error, and a tie counts as half.}
  \label{fig:campaign-stability}
\end{figure}

Let $X$ be the input of a polar approximation, let $X_T$ be its output, and let $n$ be the smaller dimension of $X$.
Let $U\Sigma V^\mathsf{T}$ be the compact SVD of $\texttt{float64}(X)$, so that $\operatorname{polar}(X)=UV^\mathsf{T}$.
We convert $X_T$ to \texttt{float64} and calculate the \emph{polar factor error} in \texttt{float64}:
\begin{equation}
 \frac{\|X_T-\operatorname{polar}(X)\|_F}{\sqrt n}.
 \label{eq:experimental-polar-error}
\end{equation}
We measure this error in three steps.
First, we generate synthetic matrices.
They are square and rectangular, with random singular vectors and geometrically spaced singular values (\Cref{app:definitions}).
Before rounding to \texttt{float16}, each matrix has unit Frobenius norm and minimum singular value $\sigma_{\min}=10^{-3}$ or $10^{-4}$.
Second, we collect Modded-NanoGPT \#91 matrices from one training run in which NorMuon uses \textsf{PE}, so only \textsf{PE} updates the model.
We randomly sample 323 of its 1,290 steps (25\%) and keep all 104 matrices that \textsf{PE} receives at each sampled step, which gives 33,592 matrices.
Third, we apply all methods to the same matrices and calculate the polar factor error of each output.
In all three panels of \Cref{fig:campaign-stability}, \textsf{PS2-55} and \textsf{PS3-55} have smaller median polar factor errors than \textsf{PE} and \textsf{GNS}.
\Cref{fig:campaign-stability} compares this error across all methods, and \Cref{app:drift} gives two more diagnostics.

\section{Conclusion and Future Work}
\label{sec:discussion}

We introduced \texttt{Polar Switch}, which replaces the quintic iterations in Muon with a few polynomial circuits of high degree.
We select these circuits and their coefficients under a FLOP budget, and our GPU implementations run them in \texttt{float16} on two or three square buffers.
In our experiments, \textsf{PS2-55} and \textsf{PS3-55} reduced runtime by 15.5--18.7\% relative to \texttt{Polar Express} on square matrices and by 8.6--16.2\% relative to \texttt{Gram Newton--Schulz} on rectangular matrices, and they gave final validation losses similar to those of established orthogonalization methods in LLM pretraining.

\paragraph{Limitations.}
Problem \eqref{eq:full-problem} is not convex.
Its constraints must hold for all inputs and all rounding factors.
Thus, our local search and interval certification use much computation.
They give certified feasible circuits, but these circuits are not necessarily globally optimal.
The number of buffer patterns increases quickly with the number $g$ of square buffers and the inner depth $k$.
Thus, for $g=3$, we searched only three fixed schedules.
For $g\ge4$, a full search is not practical.
The rounding model \eqref{eq:perturbed-circuit} is only a scalar model of the \texttt{float16} rounding.
It does not include additive store errors or buffers that stop commuting.
The circuits with inner depth $6$ satisfy its constraint, but they exceeded their bounds in training (\Cref{sec:training}).

\paragraph{Objectives of the design problem.}
Problem \eqref{eq:full-problem} maximizes the final lower endpoint for inputs in $[\ell_0,1]$.
For inputs below $\ell_0$, it only requires that each circuit, also under the rounding model \eqref{eq:perturbed-circuit}, maps $[0,1]$ into $[0,1]$.
Thus, it does not control the amplification of singular values below $\ell_0$, and $\ell_0=0.001$ may not be the best choice for all models.
We could add a slope objective to problem \eqref{eq:full-problem}, so that it both reduces the approximation error on $[\ell_0,1]$ and increases the slope $p'(0)$ of the composition $p$ to amplify very small singular values.
This slope has been maximized for Chebyshev-type polynomials~\citep{grishina2026accelerating}.
Future research could test whether this improves training.
Also, problem \eqref{eq:full-problem} maximizes the lower endpoint $\ell_T$ of the exact polynomials and uses the rounding model \eqref{eq:perturbed-circuit} only in its range constraint.
We could also require $\widetilde p_{C_t}([\ell_{t-1},1];\mathbf{v})\subseteq[\ell_t,1]$ for all $\mathbf{v}\in[s'^{-1},s']^{k_t+2}$, with a second factor $s'\in[1,s]$.
For $s'=1$, this is the first constraint of \eqref{eq:full-problem}.
For $s'>1$, $\ell_T$ would also be a lower bound under the rounding model with factor $s'$, but the largest $\ell_T$ would not increase as $s'$ increases.

\paragraph{Mixed precision storage.}
The kernel for each inner iteration computes $B_{w_\tau}\leftarrow a_\tau B_{i_\tau}B_{j_\tau}+b_\tau B_{r_\tau}+c_\tau I$.
The matrix product uses fast Tensor Cores operating in half precision, with accumulation and additions in \texttt{float32}.
The kernel then rounds the result to \texttt{float16} before storing it.
This rounding happens at every inner iteration and may be the main source of rounding error in our circuits.
We could instead store the result of each inner iteration in \texttt{float32}.
For \textsf{PS2}, two \texttt{float32} square buffers use the same memory in bytes as the four \texttt{float16} square buffers of \textsf{GNS}, without changing the FLOP count.
Reducing this rounding error could allow a smaller safety factor $s$.
\textsf{PS2} could then become more accurate and more numerically stable.

\paragraph{Circuits without diagonal shifts.}
Each inner iteration of \texttt{Polar Switch} adds $c_\tau I$.
\textsf{GNS} keeps this identity term out of the updates that track the Gram matrix, which can make these updates more numerically stable.
We could also adapt problem \eqref{eq:full-problem} to circuits without this term by setting $c_\tau=0$ in \eqref{eq:instruction} and adding the constant term only in the final update, $X\leftarrow B_{w_k}X+cX$.

\paragraph{Approximation error and training quality.}
In our experiments, a smaller polar factor error \eqref{eq:experimental-polar-error} did not give a clearly lower validation loss.
\textsf{Jordan} has the largest polar factor error (\Cref{fig:campaign-stability}), but after centering it gave losses similar to those of the other methods.
\textsf{PE (BF16)} has more than 15 times the diagonal drift and directional mixing of \textsf{PE} (\Cref{fig:drift-mixing}), but it also gave similar losses (\Cref{fig:training}).
In GPT-2 training, more iterations or an exact SVD also did not decrease the final loss~\citep{amsel2026polar,huang2026how}.
Thus, other properties of the update may determine training quality.
Candidates are the update norm (footnote~\ref{fn:jordan-centering}), the singular directions that the update amplifies, and postprocessing, such as in NorMuon~\citep{li2026normuon}.
Also, all methods use the same learning rates and weight decay, which we did not tune for each method, and the best values may differ between orthogonalization methods.
Further experiments could test these effects.

\paragraph{Integration with other Muon variants.}
Future work could replace the quintic iterations of MuonBP~\citep{khaled2026muonbp}, Hierarchical Muon~\citep{tang2026hierarchical}, and Dion3~\citep{amsel2026dion3} with \texttt{Polar Switch} circuits.
\texttt{Polar Switch} could also refresh the cached transform in CacheMuon~\citep{dev2026cachemuon}.
These combinations could reduce training time.

\section*{Acknowledgments}
This work was supported in part by the Wharton AI Fund. GPT-6 Astra and GPT-5.6 Sol were used to assist with theory and mathematical proofs, experimental design, code development, and data analysis.

\bibliographystyle{unsrtnat}
\bibliography{references}

\clearpage
\appendix
\crefalias{section}{appendix}
\crefalias{subsection}{subappendix}
\counterwithin{algorithm}{section}
\section{Solving the Polar Switch Optimization Problem}
\label{app:coefficients}

This appendix treats problem \eqref{eq:full-problem} of \Cref{sec:stage-solve}.
As in \texttt{Polar Express}, we use the initial interval $[\ell_0,1]=[0.001,1]$ after Frobenius normalization~\citep[Section~3.3]{amsel2026polar}.
In the GPT-2 Small experiments of \texttt{Polar Express}, validation loss was less sensitive to removing or reversing singular directions below $0.001\sigma_{\max}$ than below $0.01\sigma_{\max}$, where $\sigma_{\max}$ is the largest singular value of the momentum matrix~\citep[Appendix~H.1, Figure~9]{amsel2026polar}.
For $X\in\mathbb{R}^{n\times m}$ with $n\le m$, $\alpha=m/n$, and singular values in $[\ell_0,1]$, the problem uses the circuits of inner depth $k$ on $g\ge2$ square buffers with $w_0=1$ (\Cref{def:circuit,def:class}), their polynomials $p_C$ from \eqref{eq:scalar-circuit}, and their perturbed polynomials $\widetilde p_C$ from \eqref{eq:perturbed-circuit} with $s=1.005$:
\[
\begin{aligned}
 \mathcal C_k^g&=\Bigl\{C=(w_\tau,i_\tau,j_\tau,r_\tau,a_\tau,b_\tau,c_\tau)_{\tau=1}^k:\
   w_\tau\in\{1,\ldots,g\},\ i_\tau,j_\tau,r_\tau\in\{w_0,w_1,\ldots,w_{\tau-1}\},\\
 &\hphantom{{}=\Bigl\{C=(w_\tau,i_\tau,j_\tau,r_\tau,a_\tau,b_\tau,c_\tau)_{\tau=1}^k:\ }
   w_\tau\notin\{i_\tau,j_\tau\},\ a_\tau,b_\tau,c_\tau\in\mathbb{R},\ \tau=1,\ldots,k\Bigr\},\\
 \beta_1^{(0)}&=x^2,\qquad \widetilde\beta_1^{(0)}=v_0x^2,\\
 \beta_{w_\tau}^{(\tau)}&=a_\tau\beta_{i_\tau}^{(\tau-1)}\beta_{j_\tau}^{(\tau-1)}+b_\tau\beta_{r_\tau}^{(\tau-1)}+c_\tau,\\
 \widetilde\beta_{w_\tau}^{(\tau)}&=v_\tau\left(a_\tau\widetilde\beta_{i_\tau}^{(\tau-1)}\widetilde\beta_{j_\tau}^{(\tau-1)}+b_\tau\widetilde\beta_{r_\tau}^{(\tau-1)}+c_\tau\right),\\
 \beta_i^{(\tau)}&=\beta_i^{(\tau-1)},\qquad \widetilde\beta_i^{(\tau)}=\widetilde\beta_i^{(\tau-1)}\quad(i\ne w_\tau),\qquad \tau=1,\ldots,k,\\
 p_C(x)&=x\beta_{w_k}^{(k)},\\
 \widetilde p_C(x;\mathbf{v})&=v_{k+1}x\widetilde\beta_{w_k}^{(k)},\qquad \mathbf{v}=(v_0,\ldots,v_{k+1})\in[s^{-1},s]^{k+2}.
\end{aligned}
\]
An outer iteration with $C\in\mathcal C_k^g$ costs $(3\alpha+k)n^3$ leading-order FLOPs (\Cref{prop:cost}).
For a budget $f_{\max}n^3$ and a largest inner depth $k_{\max}$, problem \eqref{eq:full-problem} is
\[
\begin{aligned}
 &\max_{\substack{T\ge1,\ 1\le k_1,\ldots,k_T\le k_{\max}\\ \sum_{t=1}^T(3\alpha+k_t)\le f_{\max}}}\
 \max_{\substack{C_t\in\mathcal C_{k_t}^g,\ \ell_t\in[0,1]\\ t=1,\ldots,T}}\
 \ell_T\\[4pt]
 &\qquad\text{s.t.}\quad
 p_{C_t}([\ell_{t-1},1])\subseteq[\ell_t,1],\quad
 0\le\widetilde p_{C_t}(x;\mathbf{v})\le1,\quad x\in[0,1],\ \mathbf{v}\in[s^{-1},s]^{k_t+2},\ t=1,\ldots,T.
\end{aligned}
\]
For a feasible point of \eqref{eq:full-problem} with $\ell_T>0$, centering the output interval $[\ell_T,1]$ at one gives, in exact arithmetic, a worst-case error of at most $(1-\ell_T)/(1+\ell_T)$ in \eqref{eq:matrix-budget-problem} (\Cref{lem:centering} and \Cref{thm:relation}).

\Cref{app:greedy} reduces \eqref{eq:full-problem} to one problem per outer iteration, bounds the error after centering, and chooses the schedule, and \Cref{app:minimax-solve} solves the problem of one outer iteration numerically, rescales the circuits for \texttt{float16}, and lists all \texttt{Polar Switch} methods.

\subsection{Greedy Circuits and Optimal Schedules}
\label{app:greedy}
\label{app:search}

In problem \eqref{eq:full-problem}, the constraints on $C_t$ are $p_{C_t}([\ell_{t-1},1])\subseteq[\ell_t,1]$ and the rounding constraint, so they depend on the earlier outer iterations only through $\ell_{t-1}$.
For one outer iteration, this gives the following problem.

\begin{definition}[Per-iteration problem]
\label{def:per-iteration}
For $k\ge1$ and $\ell\in[0,1]$,
\begin{equation}
 F_k^g(\ell)=\sup_{C\in\mathcal C_k^g}\ \min_{x\in[\ell,1]}p_C(x)
 \qquad\text{subject to}\quad 0\le\widetilde p_C(y;\mathbf{v})\le1\quad(y\in[0,1],\ \mathbf{v}\in[s^{-1},s]^{k+2}).
 \label{eq:design}
\end{equation}
For a schedule $(k_1,\ldots,k_T)$, let $\ell_0^\star=\ell_0$ and $\ell_t^\star=F_{k_t}^g(\ell_{t-1}^\star)$, $t=1,\ldots,T$.
\end{definition}

The supremum is over a nonempty set, because every circuit with $a_\tau=b_\tau=c_\tau=0$ satisfies the constraint of \eqref{eq:design}.
Since $\mathbf{v}=(1,\ldots,1)$ gives $0\le p_C\le1$ on $[0,1]$, $F_k^g(\ell)\in[0,1]$.

\begin{lemma}[Monotonicity]
\label{lem:monotone}
If $0\le\ell\le\bar\ell\le1$, then $F_k^g(\ell)\le F_k^g(\bar\ell)$.
\end{lemma}
\begin{proof}
Let $0\le\ell\le\bar\ell\le1$, and let $C$ satisfy the constraint of \eqref{eq:design}.
Since $[\bar\ell,1]\subseteq[\ell,1]$, we have $\min_{x\in[\bar\ell,1]}p_C(x)\ge\min_{x\in[\ell,1]}p_C(x)$.
The constraint of \eqref{eq:design} does not depend on $\ell$, so both suprema in \eqref{eq:design} range over the same circuits.
Taking the supremum over these circuits gives $F_k^g(\bar\ell)\ge F_k^g(\ell)$.
\end{proof}

\Cref{thm:greedy} uses circuits that attain the supremum in \eqref{eq:design}.
The next lemma shows that such circuits exist.

\begin{lemma}[Attainment]
\label{lem:attain}
If $s>1$, then for $k\ge1$ and $\ell\in[0,1]$, the supremum in \eqref{eq:design} is attained.
\end{lemma}
\begin{proof}
There are finitely many buffer patterns $(w_\tau,i_\tau,j_\tau,r_\tau)_{\tau=1}^k$, and the supremum in \eqref{eq:design} is the largest of the suprema over these patterns.
Thus, we fix one pattern.
Then $C$ is determined by its coefficients $a_\tau,b_\tau,c_\tau$, $\tau=1,\ldots,k$, and it suffices to show that the supremum over these coefficients is attained.
Step 1 shows linear independence, Step 2 turns it into a bound on the coefficients, Step 3 normalizes the stored values, and Step 4 concludes by compactness.

Setup.
For $\theta=1,\ldots,k$, inner iteration $\theta$ stores $\widetilde\beta_{w_\theta}^{(\theta)}$ in \eqref{eq:perturbed-circuit}.
We say that inner iteration $0$ stores $\widetilde\beta_{w_0}^{(0)}=v_0x^2$.
For inner iteration $\tau$, let $\theta_i=\max\{\theta\in\{0,\ldots,\tau-1\}:w_\theta=i_\tau\}$, and define $\theta_j$ and $\theta_r$ in the same way with $j_\tau$ and $r_\tau$.
Then \eqref{eq:perturbed-circuit} reads
\[
 \widetilde\beta_{w_\tau}^{(\tau)}=v_\tau\gamma_\tau,
 \qquad
 \gamma_\tau=a_\tau\widetilde\beta_{w_{\theta_i}}^{(\theta_i)}\widetilde\beta_{w_{\theta_j}}^{(\theta_j)}+b_\tau\widetilde\beta_{w_{\theta_r}}^{(\theta_r)}+c_\tau .
\]
We write $\mathbb{R}[x,v_0,\ldots,v_\theta]$ for the polynomials in $x,v_0,\ldots,v_\theta$ with real coefficients, and $\mathbb{R}[x]$ when $\theta=-1$.
By induction on $\tau$, $\widetilde\beta_{w_\tau}^{(\tau)}\in\mathbb{R}[x,v_0,\ldots,v_\tau]$ has total degree at most $2^{\tau+2}-1$, and $\gamma_\tau\in\mathbb{R}[x,v_0,\ldots,v_{\tau-1}]$.
Let
\[
 \mathcal Q=\left\{f\in\mathbb{R}[x,\mathbf{v}]:\ \deg f\le2^{k+2}\right\},
 \qquad
 \Omega=[0,1]\times[s^{-1},s]^{k+2},
 \qquad
 \|f\|_\Omega=\max_{(x,\mathbf{v})\in\Omega}|f(x,\mathbf{v})|
\]
for each polynomial $f$ in $x$ and $\mathbf{v}$.
Since $s>1$, the set $\Omega$ has a nonempty interior.
A polynomial that is zero on a set with a nonempty interior is the zero polynomial, so $\|f\|_\Omega=0$ implies $f=0$.
Thus, $\|\cdot\|_\Omega$ is a norm on every finite-dimensional space of polynomials in $x$ and $\mathbf{v}$, such as $\mathcal Q$.
For $\theta=0,\ldots,k$, let
\[
 \mathcal S_\theta=\left\{v_\theta q\in\mathcal Q:\ q\in\mathbb{R}[x,v_0,\ldots,v_{\theta-1}],\ 1\le\|v_\theta q\|_\Omega\le s\right\}.
\]
Each $\mathcal S_\theta$ is the intersection of a linear subspace of $\mathcal Q$ with the closed and bounded set $\{f\in\mathcal Q:1\le\|f\|_\Omega\le s\}$.
Since $\mathcal Q$ has finite dimension, $\mathcal S_\theta$ is compact.

Step 1.
For $\rho\in\{i,j,r\}$, let $f_\rho=v_{\theta_\rho}q_\rho\ne0$ with $q_\rho\in\mathbb{R}[x,v_0,\ldots,v_{\theta_\rho-1}]$.
Assume that $f_\rho=f_{\rho'}$ whenever $\rho,\rho'\in\{i,j,r\}$ satisfy $\theta_\rho=\theta_{\rho'}$.
We show that the polynomials $f_if_j$, $f_r$, and $1$ are linearly independent, that is,
\[
 af_if_j+bf_r+c=0
 \quad\Longrightarrow\quad
 a=b=c=0 .
\]
Let $\theta=\max(\theta_i,\theta_j,\theta_r)$, and compare the coefficients of the powers of $v_\theta$.
If $\theta_r<\theta$, then only $af_if_j$ has positive degree in $v_\theta$, so $a=0$, and $bv_{\theta_r}q_r+c=0$ gives $b=c=0$.
If $\theta_r=\theta>\theta_i,\theta_j$, then only $bf_r$ has positive degree in $v_\theta$, so $b=0$, and $af_if_j+c=0$ gives $a=c=0$, because $f_if_j$ has positive degree in $v_{\max(\theta_i,\theta_j)}$.
If $\theta_r=\theta_i=\theta>\theta_j$, then $f_r=f_i$, and the coefficient of $v_\theta$ is $q_r(af_j+b)$, so $af_j+b=0$, which gives $a=b=0$ and then $c=0$.
The case $\theta_r=\theta_j=\theta>\theta_i$ follows by exchanging $i$ and $j$.
If $\theta_i=\theta_j=\theta_r=\theta$, then $f_i=f_j=f_r$, and the three terms have degrees $2$, $1$, and $0$ in $v_\theta$, so $a=b=c=0$.

Step 2.
We turn Step 1 into a bound on the coefficients.
For $f_i,f_j,f_r\in\mathcal Q$ and $a,b,c\in\mathbb{R}$, the function $\Phi(f_i,f_j,f_r,a,b,c)=\|af_if_j+bf_r+c\|_\Omega$ is continuous.
Let
\[
 \mathcal D_\tau=\left\{(f_i,f_j,f_r,a,b,c):\ f_\rho\in\mathcal S_{\theta_\rho},\ f_\rho=f_{\rho'}\text{ if }\theta_\rho=\theta_{\rho'},\ \max(|a|,|b|,|c|)=1\right\}.
\]
This set is compact, because each $\mathcal S_\theta$ is compact, the conditions $f_\rho=f_{\rho'}$ define a closed set, and $\{(a,b,c):\max(|a|,|b|,|c|)=1\}$ is compact.
On $\mathcal D_\tau$, each $f_\rho$ is nonzero and has the form of Step 1, and $(a,b,c)\ne0$.
By Step 1, $af_if_j+bf_r+c$ is a nonzero polynomial, so $\Phi>0$ on $\mathcal D_\tau$.
Thus, $\Phi$ has a positive minimum $\delta_\tau$ on $\mathcal D_\tau$.
For $(a,b,c)\ne0$, dividing $a,b,c$ by $\max(|a|,|b|,|c|)$ gives $\Phi(f_i,f_j,f_r,a,b,c)\ge\max(|a|,|b|,|c|)\,\delta_\tau$.
Thus, with $K_\tau=1/\delta_\tau$,
\[
 \max(|a|,|b|,|c|)\le K_\tau\,\|af_if_j+bf_r+c\|_\Omega
\]
for all $a,b,c\in\mathbb{R}$ and all $f_\rho\in\mathcal S_{\theta_\rho}$ with $f_\rho=f_{\rho'}$ if $\theta_\rho=\theta_{\rho'}$.
The same argument, with Step 1 for $a=0$ or $b=0$, gives such a bound for $bf_r+c$, $af_if_j+c$, and $c$, which occur when a factor or the addend is zero (Step 3).
We take $K$ as the largest of these finitely many constants, over all $\tau$.

Step 3.
We normalize the stored values without changing $p_C$ or $\widetilde p_C$.
Fix $1\le\theta<k$ and $\kappa\ne0$.
Multiply $a_\theta,b_\theta,c_\theta$ by $\kappa$.
For each $\tau>\theta$, replace $a_\tau$ by $\kappa^{-e_\tau}a_\tau$, where $e_\tau\in\{0,1,2\}$ is the number of $\rho\in\{i,j\}$ with $\theta_\rho=\theta$, and replace $b_\tau$ by $b_\tau/\kappa$ if $\theta_r=\theta$.
By induction on $\tau$, the new value stored by inner iteration $\theta$ is $\kappa\widetilde\beta_{w_\theta}^{(\theta)}$, and all other stored values do not change.
Indeed, for $\tau=\theta$ the values that it reads do not change, so the new value is $v_\theta\kappa\gamma_\theta$.
For $\tau>\theta$, each factor or addend equal to the value stored by inner iteration $\theta$ is multiplied by $\kappa$, and the new coefficients cancel this factor.
Since $\theta<k$, the value $\widetilde\beta_{w_k}^{(k)}$ does not change, so $\widetilde p_C=v_{k+1}x\widetilde\beta_{w_k}^{(k)}$ and $p_C$, its value at $\mathbf{v}=(1,\ldots,1)$, do not change.
In particular, the rescaled circuit satisfies the constraint of \eqref{eq:design}.

We apply this rescaling for $\theta=1,\ldots,k-1$ in increasing order, with $\kappa=1/\|\widetilde\beta_{w_\theta}^{(\theta)}\|_\Omega$ when $\widetilde\beta_{w_\theta}^{(\theta)}\ne0$.
Each rescaling changes only its own stored value, so afterwards $\|\widetilde\beta_{w_\theta}^{(\theta)}\|_\Omega=1$ for each $1\le\theta<k$ with $\widetilde\beta_{w_\theta}^{(\theta)}\ne0$.
If $\widetilde\beta_{w_\theta}^{(\theta)}=0$, we set to zero each $a_\tau$ with $\theta_i=\theta$ or $\theta_j=\theta$ and each $b_\tau$ with $\theta_r=\theta$.
These coefficients multiply the zero polynomial, so no stored value changes.
Since $\|\widetilde\beta_{w_0}^{(0)}\|_\Omega=\|v_0x^2\|_\Omega=s$, we have $\widetilde\beta_{w_{\theta_\rho}}^{(\theta_\rho)}\in\mathcal S_{\theta_\rho}$ for each $\tau$ and $\rho\in\{i,j,r\}$ with $\widetilde\beta_{w_{\theta_\rho}}^{(\theta_\rho)}\ne0$.
Step 2 then gives $\max(|a_\tau|,|b_\tau|,|c_\tau|)\le K\|\gamma_\tau\|_\Omega$ for $\tau=1,\ldots,k$.

It remains to bound $\|\gamma_\tau\|_\Omega$.
On $\Omega$, we have $v_\tau\ge s^{-1}$, so $|\gamma_\tau|\le s|\widetilde\beta_{w_\tau}^{(\tau)}|$ and $\|\gamma_\tau\|_\Omega\le s\|\widetilde\beta_{w_\tau}^{(\tau)}\|_\Omega$.
For $\tau<k$, $\|\widetilde\beta_{w_\tau}^{(\tau)}\|_\Omega\le1$.
For $\tau=k$, the constraint of \eqref{eq:design} gives $0\le v_{k+1}x\widetilde\beta_{w_k}^{(k)}\le1$ on $\Omega$, and $v_{k+1}\ge s^{-1}$, so $\|x\widetilde\beta_{w_k}^{(k)}\|_\Omega\le s$.
The map $f\mapsto\|xf\|_\Omega$ is also a norm on $\mathcal Q$, because $xf=0$ on $\Omega$ implies $f=0$.
All norms on the finite-dimensional space $\mathcal Q$ are equivalent, so there is $K'>0$ with $\|f\|_\Omega\le K'\|xf\|_\Omega$ for all $f\in\mathcal Q$.
Thus, $\|\widetilde\beta_{w_k}^{(k)}\|_\Omega\le K's$, and
\[
 \max(|a_\tau|,|b_\tau|,|c_\tau|)\le R,
 \qquad R=Ks\max(1,K's),
 \qquad \tau=1,\ldots,k.
\]

Step 4.
Let $\mathcal F\subseteq\mathbb{R}^{3k}$ be the set of coefficients $(a_\tau,b_\tau,c_\tau)_{\tau=1}^k$ that satisfy the constraint of \eqref{eq:design}.
For each $(x,\mathbf{v})\in\Omega$, $\widetilde p_C(x;\mathbf{v})$ is a polynomial in the coefficients, so the coefficients with $0\le\widetilde p_C(x;\mathbf{v})\le1$ form a closed set.
The set $\mathcal F$ is the intersection of these closed sets, so it is closed.
Thus, $\mathcal F\cap[-R,R]^{3k}$ is compact, and it contains the zero coefficients.
Since $p_C(x)$ is jointly continuous in $x$ and the coefficients, and $[\ell,1]$ is compact, $\min_{x\in[\ell,1]}p_C(x)$ is continuous in the coefficients.
Thus, it attains its maximum on $\mathcal F\cap[-R,R]^{3k}$.
By Step 3, every point of $\mathcal F$ has the same $p_C$ as a point of $\mathcal F\cap[-R,R]^{3k}$.
Hence, this maximum is the supremum over $\mathcal F$, and the supremum in \eqref{eq:design} is attained.
\end{proof}

The next theorem shows that, for a fixed schedule, we can solve \eqref{eq:full-problem} one outer iteration after another, as \texttt{Polar Express} shows for quintic compositions~\citep{amsel2026polar}.
No other feasible circuits for this schedule give larger lower endpoints.

\begin{theorem}[Greedy optimality]
\label{thm:greedy}
Let $(k_1,\ldots,k_T)$ be a schedule, and let $C_t\in\mathcal C_{k_t}^g$ and $\ell_t\in[0,1]$, $t=1,\ldots,T$, satisfy the constraints of \eqref{eq:full-problem}.
Then the following hold.
\begin{enumerate}
 \item[(i)] $(p_{C_t}\circ\cdots\circ p_{C_1})([\ell_0,1])\subseteq[\ell_t,1]$ for $t=1,\ldots,T$.
 \item[(ii)] $\ell_t\le\ell_t^\star$ for $t=1,\ldots,T$ (\Cref{def:per-iteration}).
 \item[(iii)] If $C_t^\star\in\mathcal C_{k_t}^g$ satisfies the constraint of \eqref{eq:design} and $\min_{x\in[\ell_{t-1}^\star,1]}p_{C_t^\star}(x)=\ell_t^\star$ for $t=1,\ldots,T$, then $C_t^\star$ and $\ell_t^\star$, $t=1,\ldots,T$, satisfy the constraints of \eqref{eq:full-problem}.
\end{enumerate}
Since $s>1$, circuits as in (iii) exist for this schedule (\Cref{lem:attain}), and they attain the optimal value $\ell_T^\star$ of \eqref{eq:full-problem} for this schedule.
We call circuits as in (iii) \emph{greedy}.
\end{theorem}
\begin{proof}
We prove (i) and (ii) by induction on $t$.
For $t=0$, the composition is the identity, so the set in (i) is $[\ell_0,1]$, and $\ell_0=\ell_0^\star$.
Assume that (i) and (ii) hold for $t-1$.

(i) The composition of the first $t-1$ outer iterations maps $[\ell_0,1]$ into $[\ell_{t-1},1]$, and the first constraint of \eqref{eq:full-problem} gives $p_{C_t}([\ell_{t-1},1])\subseteq[\ell_t,1]$, so the composition of the first $t$ outer iterations maps $[\ell_0,1]$ into $[\ell_t,1]$.

(ii) The second constraint of \eqref{eq:full-problem} is the constraint of \eqref{eq:design}, so $C_t$ is admissible in the supremum that defines $F_{k_t}^g(\ell_{t-1})$.
The first constraint gives $\ell_t\le\min_{x\in[\ell_{t-1},1]}p_{C_t}(x)\le F_{k_t}^g(\ell_{t-1})$.
Because $\ell_{t-1}\le\ell_{t-1}^\star$, \Cref{lem:monotone} gives $F_{k_t}^g(\ell_{t-1})\le F_{k_t}^g(\ell_{t-1}^\star)=\ell_t^\star$.

(iii) With $\mathbf{v}=(1,\ldots,1)$, the constraint of \eqref{eq:design} gives $p_{C_t^\star}(x)\le1$ for $x\in[0,1]$.
Together with $\min_{x\in[\ell_{t-1}^\star,1]}p_{C_t^\star}(x)=\ell_t^\star$, this gives $p_{C_t^\star}([\ell_{t-1}^\star,1])\subseteq[\ell_t^\star,1]$.

By \Cref{lem:attain}, the supremum in \eqref{eq:design} is attained at each $\ell_{t-1}^\star$, so greedy circuits exist.
By (ii), no feasible point has $\ell_T>\ell_T^\star$, and by (iii), greedy circuits attain $\ell_T^\star$.
\end{proof}

By \Cref{thm:greedy}(i), the composition of all $T$ outer iterations maps $[\ell_0,1]$ into $[\ell_T,1]$.
The next lemma gives the constant that centers $[\ell_T,1]$ at one.

\begin{lemma}[Final centering]
\label{lem:centering}
For $\ell\in(0,1]$,
\[
 \min_{\lambda>0}\ \max_{y\in[\ell,1]}|1-\lambda y|=\frac{1-\ell}{1+\ell},
\]
and $\lambda=2/(1+\ell)$ attains the minimum.
This minimum strictly decreases in $\ell$.
\end{lemma}
\begin{proof}
Fix $\lambda>0$.
The function $y\mapsto|1-\lambda y|$ is convex, so its maximum on $[\ell,1]$ is attained at an endpoint.
Since $\ell\le1$, we have $1-\lambda\ell\ge1-\lambda$ and $\lambda-1\ge\lambda\ell-1$, so this maximum is $\max(1-\lambda\ell,\lambda-1)$.
The first term decreases and the second term increases in $\lambda$, so their maximum is smallest when they are equal, that is, when $\lambda(1+\ell)=2$.
Then both terms equal $(1-\ell)/(1+\ell)=2/(1+\ell)-1$, which strictly decreases in $\ell$.
\end{proof}

The next proposition bounds the error of the centered output of any feasible point of \eqref{eq:full-problem}.

\begin{proposition}[Error bound after centering]
\label{thm:relation}
Let $C_t$ and $\ell_t$, $t=1,\ldots,T$, satisfy the constraints of \eqref{eq:full-problem} with $\ell_T>0$, and let $\widehat C_T$ be $C_T$ with the coefficients of its last inner iteration multiplied by $\lambda=2/(1+\ell_T)$.
Then, in exact arithmetic, $\widehat p=p_{\widehat C_T}\circ p_{C_{T-1}}\circ\cdots\circ p_{C_1}$ satisfies $\max_{x\in[\ell_0,1]}|1-\widehat p(x)|\le(1-\ell_T)/(1+\ell_T)$, and $\|{\operatorname{polar}(X)-\widehat p(X)}\|_2\le(1-\ell_T)/(1+\ell_T)$ for every $X\in\mathbb{R}^{n\times m}$ with singular values in $[\ell_0,1]$.
\end{proposition}
\begin{proof}
The last inner iteration of $C_T$ is linear in $a_{k_T},b_{k_T},c_{k_T}$, and no later inner iteration reads its output.
Thus, $\beta_{w_{k_T}}^{(k_T)}$ and $p_{C_T}$ are multiplied by $\lambda$, so $p_{\widehat C_T}=\lambda p_{C_T}$.
By \Cref{thm:greedy}(i), $p_{C_T}\circ\cdots\circ p_{C_1}$ maps $[\ell_0,1]$ into $[\ell_T,1]$, so $\widehat p$ maps it into $[\lambda\ell_T,\lambda]$.
By \Cref{lem:centering}, $\max_{x\in[\ell_0,1]}|1-\widehat p(x)|\le\max_{y\in[\ell_T,1]}|1-\lambda y|=(1-\ell_T)/(1+\ell_T)$.
The bound on $\|{\operatorname{polar}(X)-\widehat p(X)}\|_2$ follows as in the proof of \Cref{prop:scalar}.
\end{proof}

Thus, for each schedule, greedy circuits attain the largest $\ell_T$ in \eqref{eq:full-problem}, and centering their output gives the error bound $(1-\ell_T^\star)/(1+\ell_T^\star)$, which strictly decreases in $\ell_T^\star$ (\Cref{lem:centering}).
It remains to choose the schedule with $\sum_{t=1}^T(3\alpha+k_t)\le f_{\max}$ and $k_t\le k_{\max}$ that maximizes $\ell_T^\star$, by dynamic programming over the cost, as in~\citep{lee2022minimax}.

\begin{proposition}[Optimal schedule under a FLOP budget]
\label{thm:budget}
Let $0<\ell_0<1$, $\alpha\ge1$, and $3\alpha+1\le f_{\max}<\infty$.
For a cost $f$, let $L(f)$ be the largest $\ell_T^\star$ over the schedules with $\sum_{t=1}^T(3\alpha+k_t)=f$ and $k_t\le k_{\max}$, and let $L(0)=\ell_0$.
Then
\[
 L(f)=\max_{1\le k\le k_{\max}}F_k^g\left(L(f-3\alpha-k)\right),
\]
over the $k$ for which $L(f-3\alpha-k)$ is defined.
A schedule that attains $\max_{0<f\le f_{\max}}L(f)$ maximizes $\ell_T^\star$ subject to $\sum_{t=1}^T(3\alpha+k_t)\le f_{\max}$ and $k_t\le k_{\max}$.
With greedy circuits, this schedule solves \eqref{eq:full-problem}.
\end{proposition}
\begin{proof}
Each outer iteration costs $3\alpha+k_t\ge3\alpha+1>0$.
Thus, a schedule of cost $f$ has $T\le f/(3\alpha+1)$, and with $k_t\le k_{\max}$ there are finitely many such schedules, so each $L(f)$ is a maximum.

We first show that $L(f)$ is at most the right-hand side.
Let $(k_1,\ldots,k_T)$ have cost $f>0$.
Removing $k_T$ gives a schedule of cost $f-3\alpha-k_T$ whose lower endpoint satisfies $\ell_{T-1}^\star\le L(f-3\alpha-k_T)$.
If $T=1$, this is the empty schedule with $\ell_0^\star=L(0)$.
By \Cref{lem:monotone}, $\ell_T^\star=F_{k_T}^g(\ell_{T-1}^\star)\le F_{k_T}^g(L(f-3\alpha-k_T))$, so $L(f)$ is at most the right-hand side.

Conversely, for each admissible $k$, appending $k$ to a schedule that attains $L(f-3\alpha-k)$ gives a schedule of cost $f$ with $\ell_T^\star=F_k^g(L(f-3\alpha-k))$.
Hence, $L(f)$ is at least the right-hand side.

Since $3\alpha+k>0$, the recursion expresses $L(f)$ through values of $L$ at smaller costs.
Thus, we can compute $L(f)$ for all costs $f\le f_{\max}$ in increasing order and keep, for each cost, a schedule that attains $L(f)$.
A schedule with the largest $L(f)$ over $0<f\le f_{\max}$ maximizes $\ell_T^\star$ over all feasible schedules.
By \Cref{thm:greedy}, the supremum of $\ell_T$ in \eqref{eq:full-problem} for each schedule is $\ell_T^\star$, so this schedule with greedy circuits solves \eqref{eq:full-problem}.
\end{proof}

\subsection{Numerical Search and Coefficients}
\label{app:minimax-solve}

\Cref{app:greedy} reduces \eqref{eq:full-problem} to the per-iteration problem \eqref{eq:design}.
We solve \eqref{eq:design} numerically, for one inner depth $k$ and one input interval $[\ell,1]$ at a time.
For $g=2$, we use this search in place of each $F_k^g$ in the recursion of \Cref{thm:budget}.
For $g=3$, we fix the schedules $(5,5)$, $(6,4)$, and $(5,4)$, and we apply the search to their outer iterations in order, starting from $[\ell_0,1]$ (\Cref{sec:design}).

\paragraph{Sampled problem.}
We call $(w_\tau,i_\tau,j_\tau,r_\tau)_{\tau=1}^k$ the buffer pattern of $C\in\mathcal C_k^g$.
For each buffer pattern, the search optimizes only $a_\tau,b_\tau,c_\tau$, $\tau=1,\ldots,k$.
For $g=2$, it tries all $2^{k-1}$ alternating buffer patterns, for example 32 for $k=6$ (\Cref{sec:search}).
For $g=3$, it first removes some patterns (degree-bound pruning below).
Problem \eqref{eq:design} has a constraint for every $y\in[0,1]$ and every $\mathbf{v}$, so the search replaces $[\ell,1]$, $[0,1]$, and $[s^{-1},s]^{k+2}$ with finite sets $\mathcal X\subseteq[\ell,1]$, $\mathcal Y\subseteq[0,1]$, and $\mathcal V\subseteq[s^{-1},s]^{k+2}$, and it computes a local solution of
\begin{equation}
 \max_{a_\tau,b_\tau,c_\tau\in\mathbb{R},\ \tau=1,\ldots,k}\ \min_{x\in\mathcal X}p_C(x)
 \qquad\text{subject to}\quad 0\le\widetilde p_C(y;\mathbf{v})\le1\quad(y\in\mathcal Y,\ \mathbf{v}\in\mathcal V).
 \label{eq:sampled}
\end{equation}
At the start, $\mathcal X$ contains 161 geometrically and 161 uniformly spaced points in $[\ell,1]$, $\mathcal Y$ also contains 33 uniformly spaced points in $[0,\ell]$, and $\mathcal V$ contains random points and corners of $[s^{-1},s]^{k+2}$.
The search solves \eqref{eq:sampled} by local optimization from several starting coefficients, which are random or come from earlier runs.
Each start gives one candidate circuit.

\paragraph{Adversarial points.}
A candidate can violate the constraint of \eqref{eq:design} between the points of $\mathcal X$, $\mathcal Y$, and $\mathcal V$.
Thus, at regular intervals during the local optimization, the search finds the inputs and rounding factors where $p_C$ is smallest on $[\ell,1]$ and where $\widetilde p_C$ is nearest to $0$ or $1$.
It finds these points by local minimization from the worst points of $\mathcal X$, $\mathcal Y$, and $\mathcal V$, and adds them to these sets.

\paragraph{Certification.}
Finite sets cannot show the constraint of \eqref{eq:design} for all inputs and rounding factors.
Thus, we accept a candidate only after interval arithmetic with outward rounding shows that
\[
 p_C([\ell,1])\subseteq[\ell',1],
 \qquad
 \widetilde p_C\left([0,1];[s^{-1},s]^{k+2}\right)\subseteq[0,1],
 \qquad
 \ell'>0.
\]
We say that this check certifies the candidate.
Interval arithmetic computes intervals that contain all exact values of $p_C$ and $\widetilde p_C$, and outward rounding keeps them valid in floating-point arithmetic.
Certification is expensive, so we certify only the best candidates on $\mathcal X$, $\mathcal Y$, and $\mathcal V$.

\paragraph{Refinement.}
Certification also gives several points: the points in the whole domain where it finds $p_C$ smallest on $[\ell,1]$ and $\widetilde p_C$ nearest to $0$ or $1$.
The search adds these points to $\mathcal X$, $\mathcal Y$, and $\mathcal V$, runs the local optimization again from the candidate, and certifies the new candidate.
It repeats this for a few rounds, until $\ell'$ stops increasing.
Over all buffer patterns and starting coefficients, the search returns the accepted candidate with the largest $\ell'$, and the search for the next outer iteration uses $[\ell',1]$.
After the rescaling for \texttt{float16} below, we certify the rescaled circuit again.

\paragraph{Degree-bound pruning.}
\label{app:degree-pruning}
For $g=3$, the search removes buffer patterns before the local optimization.
The buffer pattern bounds the degree of $p_C$.
A polynomial of bounded degree that satisfies the constraint of \eqref{eq:design} and is zero at zero cannot increase arbitrarily fast, so inequalities for bounded polynomials, such as the Bernstein--Szeg\H{o} inequality~\citep{milev2005bernstein}, bound $p_C(\ell)$.
The search removes a pattern when this bound is at most the certified $\ell'$ of the best circuit found so far for the same $k$ and $[\ell,1]$.

\paragraph{Rescaling for \texttt{float16}.}
\label{app:intervals}
In exact arithmetic, each outer iteration before centering maps $[0,1]$ into $[0,1]$, but the stored values $\beta_{w_\tau}^{(\tau)}(x)$ can be very large or very small.
We can rescale them, because each $\beta_{w_\tau}^{(\tau)}$ before the last inner iteration has a free scale.
For example, the first two inner iterations of \textsf{PS2-55} in \Cref{alg:polar-switch-methods} are, with coefficients rounded to four significant digits,
\[
 B\leftarrow A^2-1.039A+0.001138I,
 \qquad
 A\leftarrow164.1B^2+49.58B+0.0005001I.
\]
For $\kappa\ne0$, replacing them by
\[
 B\leftarrow\kappa\left(A^2-1.039A+0.001138I\right),
 \qquad
 A\leftarrow\frac{164.1}{\kappa^2}B^2+\frac{49.58}{\kappa}B+0.0005001I
\]
multiplies the first stored value by $\kappa$ and does not change the second, so $p_C$ and $\widetilde p_C$ do not change.
Step 3 of the proof of \Cref{lem:attain} gives this rescaling for any circuit.
Before we use the circuits, we choose the scales so that a certified upper bound on $\max_{x\in[0,1]}|\beta_{w_\tau}^{(\tau)}(x)|$ lies in $[1/16,1024]$ for each $\tau=1,\ldots,k-1$.
The rescaling does not change $\widetilde p_C$, so the rescaled circuits still satisfy the rounding constraint of \eqref{eq:design} before centering.

\paragraph{\texttt{Polar Switch} methods.}
\Cref{alg:polar-switch-methods} gives all \texttt{Polar Switch} methods after this rescaling, with the buffers $B_1,B_2,B_3$ of \Cref{alg:standard} written as $A,B,C$ and with each coefficient $\pm1$ written as its sign only.
\Cref{sec:training} recommends inner depths of at most five.
The comment after each $A\gets XX^\mathsf{T}$ shows how the outer iteration maps the singular values in exact arithmetic, starting from $[0.001,1]$, and the last outer iteration includes the centering factor $2/(1+\ell_T)$ of \Cref{thm:relation}.
The rounding constraint keeps each exact output below one, so the final intervals are not symmetric about one.

\begin{breakablealgorithm}
\caption{Circuits and coefficients of the \texttt{Polar Switch} methods. We recommend \textsf{PS2-55} and \textsf{PS3-55}.}
\label{alg:polar-switch-methods}
\begin{algorithmic}[1]
\Require $X\in\mathbb{R}^{n\times m}$, $n\le m$; a \texttt{Polar Switch} method
\Ensure Approximation to the polar factor of $X$
\State $X\gets\texttt{float32}(X);\quad X\gets X/(\lVert X\rVert_F+10^{-7});\quad X\gets\texttt{float16}(X)$
\Switch{the method}
\Case{\textsf{PS2-55}}
\State $A\gets XX^\mathsf{T}$ \Comment{Outer iteration 1: $[0.001,1]\to[0.045888,0.979992]$}
\State $B\gets A^2 - 1.039372688617118A + 0.0011376825939118764I$
\State $A\gets 164.07644513428048B^2 + 49.58061223879703B + 0.000500089059050944I$
\State $B\gets A^2 + 4.238519015111072A + 2.9096036081096996I$
\State $A\gets B^2 - 0.7099643179341211A - 3.188845137257469I$
\State $B\gets A^2 + 0.10116427259090495B + 0.36987195573416237I$
\State $X\gets BX$
\State $A\gets XX^\mathsf{T}$ \Comment{Outer iteration 2: $[0.045888,0.979992]\to[0.872263,1.106236]$}
\State $B\gets 3.6023274985419302A^2 - 6.053717133712312A + 0.03621204508077079I$
\State $A\gets B^2 + 2.8221769622596162B + 1.4534435713646865I$
\State $B\gets A^2 + 0.5287633553256361B - 0.15855532053714858I$
\State $A\gets B^2 - 0.3880284725332113A - 0.745819141374713I$
\State $B\gets A^2 + 2.072314176413394B + 3.287628403409071I$
\State $X\gets BX$
\EndCase
\Case{\textsf{PS2-64}}
\State $A\gets XX^\mathsf{T}$ \Comment{Outer iteration 1: $[0.001,1]\to[0.078772,0.983217]$}
\State $B\gets 11.497164886604264A^2 - 12.511452967934748A + 0.004636590213080262I$
\State $A\gets B^2 + 3.573449564442145B + 0.0001482458025758612I$
\State $B\gets A^2 + 3.7647709198378587A + 0.0005602863971002667I$
\State $A\gets B^2 + 4.406111844100261B + 3.3624041947754058I$
\State $B\gets A^2 - 1.2242502595103495B - 4.464869468776392I$
\State $A\gets B^2 + 0.28137156744401426A + 1.0392921163482471I$
\State $X\gets AX$
\State $A\gets XX^\mathsf{T}$ \Comment{Outer iteration 2: $[0.078772,0.983217]\to[0.897110,1.081213]$}
\State $B\gets 5.598448569756209A^2 - 6.500656412649984A + 1.4867267155400337I$
\State $A\gets B^2 - 0.9846462094765358A - 1.917580856812255I$
\State $B\gets A^2 + 3.2502877930209344A + 1.959088303306074I$
\State $A\gets B^2 + 1.4407850437758025A + 4.709485957934938I$
\State $X\gets AX$
\EndCase
\Case{\textsf{PS2-54}}
\State $A\gets XX^\mathsf{T}$ \Comment{Outer iteration 1: $[0.001,1]\to[0.045888,0.979992]$}
\State $B\gets A^2 - 1.039372688617118A + 0.0011376825939118764I$
\State $A\gets 164.07644513428048B^2 + 49.58061223879703B + 0.000500089059050944I$
\State $B\gets A^2 + 4.238519015111072A + 2.9096036081096996I$
\State $A\gets B^2 - 0.7099643179341211A - 3.188845137257469I$
\State $B\gets A^2 + 0.10116427259090495B + 0.36987195573416237I$
\State $X\gets BX$
\State $A\gets XX^\mathsf{T}$ \Comment{Outer iteration 2: $[0.045888,0.979992]\to[0.771715,1.201363]$}
\State $B\gets 6.901850900116884A^2 - 7.814674292160296A + 1.6529257960327404I$
\State $A\gets B^2 - 0.9827284771159464A - 2.485618081012963I$
\State $B\gets A^2 + 3.89928636441674A + 2.691162325834262I$
\State $A\gets B^2 + 1.1933903228025404A + 4.548823708141432I$
\State $X\gets AX$
\EndCase
\Case{\textsf{PS2-512}}
\State $A\gets XX^\mathsf{T}$ \Comment{Outer iteration 1: $[0.001,1]\to[0.045888,0.979992]$}
\State $B\gets -12.809232808247291A^2 + 13.313566743030574A - 0.014572841207339851I$
\State $A\gets -B^2 + 3.870693349203208B - 0.0005000890581930987I$
\State $B\gets -A^2 + 4.238519015153693A - 2.909603608169363I$
\State $A\gets B^2 + 0.7099643179547215A - 3.1888451373814886I$
\State $B\gets A^2 - 0.1011642725968188B + 0.3698719557629632I$
\State $X\gets BX$
\State $A\gets XX^\mathsf{T}$ \Comment{Outer iteration 2: $[0.045888,0.979992]\to[0.183165,0.987609]$}
\State $B\gets 7.900332287031677A^2 - 10.949298810612854A + 4.014586689560672I$
\State $X\gets BX$
\State $A\gets XX^\mathsf{T}$ \Comment{Outer iteration 3: $[0.183165,0.987609]\to[0.894082,1.091006]$}
\State $B\gets -6.169925115269519A^2 + 7.663095425175341A - 1.9889058140650846I$
\State $A\gets B^2 - 1.0658391416925121A + 1.8937604780225805I$
\State $X\gets AX$
\EndCase
\Case{\textsf{PS2-4111}}
\State $A\gets XX^\mathsf{T}$ \Comment{Outer iteration 1: $[0.001,1]\to[0.024769,0.957717]$}
\State $B\gets -A^2 + 1.1102700886611114A - 0.004239199255238769I$
\State $A\gets -146.1291296102139B^2 + 48.29423778494536B - 2.470081603735498I$
\State $B\gets -A^2 - 4.60219907694991B + 6.153664831948212I$
\State $A\gets B^2 - 7.931721104744646B + 15.885015057782338I$
\State $X\gets AX$
\State $A\gets XX^\mathsf{T}$ \Comment{Outer iteration 2: $[0.024769,0.957717]\to[0.101332,0.987609]$}
\State $B\gets 8.471969885110855A^2 - 11.606992370953925A + 4.098200452703165I$
\State $X\gets BX$
\State $A\gets XX^\mathsf{T}$ \Comment{Outer iteration 3: $[0.101332,0.987609]\to[0.375000,0.987609]$}
\State $B\gets 6.524428222537468A^2 - 9.349022444806094A + 3.7960064446656627I$
\State $X\gets BX$
\State $A\gets XX^\mathsf{T}$ \Comment{Outer iteration 4: $[0.375000,0.987609]\to[0.930600,1.056149]$}
\State $B\gets 2.599619231760616A^2 - 4.624248122492742A + 3.0804743574608175I$
\State $X\gets BX$
\EndCase
\Case{\textsf{PS2-11111}}
\State $A\gets XX^\mathsf{T}$ \Comment{Outer iteration 1: $[0.001,1]\to[0.004191,0.987609]$}
\State $B\gets 9.141188870286298A^2 - 12.372497198168581A + 4.191604898553078I$
\State $X\gets BX$
\State $A\gets XX^\mathsf{T}$ \Comment{Outer iteration 2: $[0.004191,0.987609]\to[0.017516,0.987609]$}
\State $B\gets 9.049938391659953A^2 - 12.268381113002091A + 4.179133497635361I$
\State $X\gets BX$
\State $A\gets XX^\mathsf{T}$ \Comment{Outer iteration 3: $[0.017516,0.987609]\to[0.072222,0.987609]$}
\State $B\gets 8.673507075405423A^2 - 11.838007913719911A + 4.126813472556659I$
\State $X\gets BX$
\State $A\gets XX^\mathsf{T}$ \Comment{Outer iteration 4: $[0.072222,0.987609]\to[0.278599,0.987609]$}
\State $B\gets 7.222927473854947A^2 - 10.16475257798277A + 3.910313522197052I$
\State $X\gets BX$
\State $A\gets XX^\mathsf{T}$ \Comment{Outer iteration 5: $[0.278599,0.987609]\to[0.867287,1.118678]$}
\State $B\gets 3.8847023442663615A^2 - 6.353149203553178A + 3.582738700662502I$
\State $X\gets BX$
\EndCase
\Case{\textsf{PS3-55}}
\State $A\gets XX^\mathsf{T}$ \Comment{Outer iteration 1: $[0.001,1]\to[0.046460,0.978218]$}
\State $B\gets -A^2 + 1.0372630529657008A - 0.0011080641835321228I$
\State $C\gets -169.7011946136095B^2 + 50.90200678616221B - 0.0004999998950287138I$
\State $B\gets -C^2 + 4.259459173066772C - 2.9027674651455544I$
\State $C\gets B^2 + 0.6326518033906352C - 3.1055845493128005I$
\State $A\gets C^2 - 0.5157563645706955A + 0.5614111759107233I$
\State $X\gets AX$
\State $A\gets XX^\mathsf{T}$ \Comment{Outer iteration 2: $[0.046460,0.978218]\to[0.898309,1.080954]$}
\State $B\gets 7.281110320128563A^2 - 7.974535412410232A + 0.041676558977155734I$
\State $C\gets B^2 + 2.6715673993971243B + 1.495553622588639I$
\State $A\gets C^2 - 1.103759381672208A - 0.3426296621352286I$
\State $B\gets A^2 - 0.39530569914337854B - 1.172904844980118I$
\State $A\gets B^2 + 2.6713299568702316A + 3.4991981111096626I$
\State $X\gets AX$
\EndCase
\Case{\textsf{PS3-64}}
\State $A\gets XX^\mathsf{T}$ \Comment{Outer iteration 1: $[0.001,1]\to[0.084881,0.983377]$}
\State $B\gets 13.61303780449359A^2 - 13.826546610078044A + 0.004416673108912614I$
\State $C\gets B^2 + 3.673535324598937B + 0.0001452487295876094I$
\State $B\gets C^2 + 3.8954801610559158C + 0.0005680244195404851I$
\State $C\gets B^2 + 4.454317479650682B + 3.3893440889745525I$
\State $B\gets C^2 - 1.0853078608959024B - 4.366629679268323I$
\State $A\gets B^2 - 1.7963370170324888A + 1.8693023230294266I$
\State $X\gets AX$
\State $A\gets XX^\mathsf{T}$ \Comment{Outer iteration 2: $[0.084881,0.983377]\to[0.913241,1.070202]$}
\State $B\gets -6.592451443613544A^2 + 8.239733131416537A - 0.15364948278986565I$
\State $C\gets -B^2 + 2.40570058353004B - 0.8951407981135736I$
\State $B\gets -C^2 + 0.942592113051982B - 1.5885457840104837I$
\State $A\gets B^2 - 1.1355357330585019A + 2.0140727879667097I$
\State $X\gets AX$
\EndCase
\Case{\textsf{PS3-54}}
\State $A\gets XX^\mathsf{T}$ \Comment{Outer iteration 1: $[0.001,1]\to[0.046460,0.978218]$}
\State $B\gets -A^2 + 1.0372630529657008A - 0.0011080641835321228I$
\State $C\gets -169.7011946136095B^2 + 50.90200678616221B - 0.0004999998950287138I$
\State $B\gets -C^2 + 4.259459173066772C - 2.9027674651455544I$
\State $C\gets B^2 + 0.6326518033906352C - 3.1055845493128005I$
\State $A\gets C^2 - 0.5157563645706955A + 0.5614111759107233I$
\State $X\gets AX$
\State $A\gets XX^\mathsf{T}$ \Comment{Outer iteration 2: $[0.046460,0.978218]\to[0.800196,1.178895]$}
\State $B\gets -8.570223166079066A^2 + 9.830909786264433A - 1.8649399525020482I$
\State $C\gets -B^2 - 1.337706412340474A + 1.4220365873153455I$
\State $B\gets -C^2 + 2.184691275668604A + 0.32160367575037047I$
\State $A\gets B^2 - 9.409860595990587A + 4.317539075329395I$
\State $X\gets AX$
\EndCase
\EndSwitch
\State \Return $X$
\end{algorithmic}
\end{breakablealgorithm}

\section{Experimental Settings and Additional Results}
\label{app:campaign}

Each run uses one NVIDIA B200 GPU with PyTorch~2.14.0 and CUDA~13.0.
\texttt{Polar Switch}, \textsf{Jordan}, and \textsf{PE} use the symmetric and rectangular matrix multiplication kernels of QuACK, from the package \texttt{quack-kernels}\footnote{\url{https://github.com/Dao-AILab/quack/tree/10ce505a7ecf68e46753009bd08dcf9bd6c059ee}}~0.6.5.
We extend their epilogue with NVIDIA's CUTLASS\footnote{\url{https://github.com/NVIDIA/cutlass/tree/098de2a652cf8f00fd70b2df54051c7eccbb855a}} domain-specific language (DSL)~4.8.0 to add optional diagonal shifts, through a PyTorch custom operator.
QuACK requires both matrix dimensions to be at least 256 and multiples of 8.
For smaller matrices, such as the $192\times768$ matrices of Modded-NanoGPT \#91, we fall back to a Triton~\citep{tillet2019triton} implementation of the same method, with Triton\footnote{\url{https://github.com/triton-lang/triton/tree/c01b6774b1865984607d89d89d3a10833de92037}}~3.8.0.
\textsf{GNS} uses the official \texttt{Gram Newton--Schulz} code~\citep{zhang2026gram}.\footnote{\url{https://github.com/Dao-AILab/gram-newton-schulz/tree/e45d0aca7083cb275c9a303220c05c4abecd9187}}

\Cref{alg:polar-switch-methods} gives the \texttt{Polar Switch} methods, and \Cref{tab:baseline-coefficients} gives the coefficients of the baselines.
\textsf{Jordan} uses the fixed quintic coefficients of the Muon implementation~\citep{jordan2024muon} with 5 iterations.
\textsf{PE} uses the coefficients of the official \texttt{Polar Express} code~\citep{amsel2026polar}\footnote{\url{https://github.com/NoahAmsel/PolarExpress/blob/71cc37943d99cae780024c1d198977f2f8795407/polar_express.py}} with 5 iterations.\footnote{The code computes these coefficients with a cushion and a safety factor, both $0.02$. With the cushion, it fits each quintic on $[\max(\ell,0.02u),u]$ instead of the current interval $[\ell,u]$, and then scales it so that $p(\ell)+p(u)=2$. With the safety factor, it replaces each quintic $p(x)$ except the last one by $p(x/1.02)$.}
\textsf{PE (BF16)} uses the same coefficients with \texttt{bfloat16} storage and products.
\textsf{Jordan}, \textsf{PE}, and \textsf{PE (BF16)} follow \Cref{alg:baselines}.
\textsf{GNS} uses the five coefficient triples of the official code, restarts after its second quintic iteration, and centers the output of its last iteration at one.
In training, \textsf{GNS} falls back to \textsf{PE} for square matrices, as in the official code.
\begin{table}[htbp]
\centering
\caption{Baseline coefficients for $p_t(x)=c_tx+b_tx^3+a_tx^5$, including the safety factor of each method. \textsf{Jordan} and \textsf{GNS} are centered at the last iteration.\label{tab:baseline-coefficients}}
\begin{tabular}{@{}lcrrr@{}}
\toprule
Method & Iteration $t$ & $c_t$ & $b_t$ & $a_t$ \\ \midrule
\multirow{2}{*}{\textsf{Jordan}} & 1--4 & $3.4445$ & $-4.775$ & $2.0315$ \\
 & 5 & $4.117967776458191$ & $-5.708606802899656$ & $2.4286983706996126$ \\
\midrule
\multirow{5}{*}{\textsf{PE}} & 1 & $8.156554524902461$ & $-22.483292925577953$ & $15.878769915207457$ \\
 & 2 & $4.042929935166734$ & $-2.8089174659087113$ & $0.500017845105131$ \\
 & 3 & $3.891667802292659$ & $-2.772484153217684$ & $0.506064817850339$ \\
 & 4 & $3.2857536577556523$ & $-2.3681294933425354$ & $0.4644902423300304$ \\
 & 5 & $2.3465413258596395$ & $-1.709782838268708$ & $0.42323551169305257$ \\
\midrule
\multirow{5}{*}{\textsf{GNS}} & 1 & $7.892582874424408$ & $-20.38301394587957$ & $13.555306149406924$ \\
 & 2 & $3.911484868135431$ & $-2.5464635929060884$ & $0.4268988319673074$ \\
 & 3 & $3.760657955697423$ & $-2.512819018216563$ & $0.4323647349070073$ \\
 & 4 & $3.160399673686287$ & $-2.149649518898498$ & $0.3996366907664389$ \\
 & 5 & $2.365477637327074$ & $-1.5563980023821433$ & $0.3542623265416755$ \\
\bottomrule
\end{tabular}
\end{table}

\begin{algorithm}[htbp]
\caption{Quintic baselines \textsf{Jordan}, \textsf{PE}, and \textsf{PE (BF16)}}
\label{alg:baselines}
\begin{algorithmic}[1]
\Require $X\in\mathbb{R}^{n\times m}$, $n\le m$; coefficients $(a_t,b_t,c_t)_{t=1}^5$ from \Cref{tab:baseline-coefficients}; precision $\in\{\texttt{float16},\texttt{bfloat16}\}$
\Ensure Approximation to the polar factor of $X$
\If{precision is \texttt{float16}} \Comment{\textsf{Jordan} and \textsf{PE}, as in \Cref{alg:standard}}
  \State $X\gets\texttt{float32}(X);\quad X\gets X/(\lVert X\rVert_F+10^{-7});\quad X\gets\texttt{float16}(X)$
\Else \Comment{\textsf{PE (BF16)}}
  \State $X\gets\texttt{bfloat16}\left(X/(\lVert X\rVert_F+0.01)/1.02\right)$
\EndIf
\For{$t=1,\ldots,5$}
  \State $A\gets XX^\mathsf{T}$; $Z\gets b_tA+a_tA^2$; $X\gets ZX+c_tX$ \Comment{Stored in the chosen precision}
\EndFor
\State \Return $X$
\end{algorithmic}
\end{algorithm}

\subsection{LLM Experiments}
\label{app:training}

\paragraph{Modded-NanoGPT \#91.}
We use the training script of record \#91 of the Modded-NanoGPT short track~\citep{jordan2024moddednanogpt}.\footnote{\url{https://github.com/KellerJordan/modded-nanogpt/blob/bc3a0c2d640d0d73dedaef87eae26148d2e32afb/train_gpt.py}}
The model has 11 blocks, width 768, and 6 attention heads.
The original script uses FlashAttention-3, and we use FlashAttention-4~\citep{zadouri2026flashattention4}\footnote{\url{https://github.com/Dao-AILab/flash-attention/tree/1bda8f9290cd48d030f1516f0e680cd464ef3554}} to run it on the B200 GPU.
We train on GPT-2-tokenized FineWeb~\citep{penedo2024fineweb} for 1,250 scheduled steps followed by 40 extension steps.
NorMuon~\citep{li2026normuon} updates hidden matrix weights with base learning rate $0.023$,
and Adam updates the remaining parameters with base learning rate $0.008$.
We retain the script's momentum, parameter-specific learning rates, weight decay, and learning-rate schedule.
All methods use seeds 0--19.
Final validation loss is measured on the first 10,485,760 tokens of the
GPT-2-tokenized FineWeb validation shard.

\paragraph{Qwen3-0.6B.}
We use the Qwen3-0.6B~\citep{qwen2025qwen3} architecture with random initialization.
Training uses FineWeb-Edu~\citep{penedo2024fineweb} \texttt{sample-100BT},
using the \texttt{Qwen/Qwen3-0.6B} tokenizer, for 22,889 steps of 524,288 tokens.
Each step accumulates gradients over 8 microbatches of 32 sequences of length 2,048.
Attention uses the fused attention kernel of cuDNN~\citep{chetlur2014cudnn} (\texttt{nvidia-cudnn-frontend}~1.30.0)\footnote{\url{https://github.com/NVIDIA/cudnn-frontend/tree/42286a6fb8792055f3db80403350e80114ea086c}}, and the gated MLP uses the fused MLP kernel of QuACK.
Muon updates hidden matrix weights, and AdamW updates other parameters with
moment decay rates $0.9$ and $0.95$.
Both optimizers have weight decay $0.1$, and we clip the gradient norm at 1.
The learning rate increases linearly
to $0.0015$ during the first 1\% of training, then follows cosine decay to $0.00015$.
Muon uses momentum $0.95$ and scales each matrix update by
$0.2\sqrt m$, where $m$ is the larger dimension of the matrix.
Validation uses a fixed sequence of 8,388,608 held-out FineWeb-Edu tokens.

\paragraph{Gemma 3-1B.}
We use the Gemma 3-1B~\citep{gemma2025gemma3} architecture with random initialization.
Training uses FineWeb-Edu~\citep{penedo2024fineweb} \texttt{sample-100BT},
using the \texttt{google/gemma-3-1b-pt} tokenizer, for 41,962 steps of 524,288 tokens.
Each step accumulates gradients over 8 microbatches of 32 sequences of length 2,048.
Attention uses the fused attention kernel of cuDNN~\citep{chetlur2014cudnn} (\texttt{nvidia-cudnn-frontend}~1.30.0), and the gated MLP uses the fused MLP kernel of QuACK.
Muon updates hidden matrix weights, and AdamW updates other parameters with
moment decay rates $0.9$ and $0.95$.
Both optimizers have weight decay $0.1$, and we clip the gradient norm at 1.
The learning rate increases linearly
to $0.0005$ during the first 1\% of training, then follows cosine decay to $0.00005$.
Muon uses momentum $0.95$ and scales each matrix update by
$0.2\sqrt m$, where $m$ is the larger dimension of the matrix.
Validation uses a fixed sequence of 8,388,608 held-out FineWeb-Edu tokens.

\subsection{Runtime and Approximation Accuracy}
\label{app:definitions}

\paragraph{Test matrices.}
\label{app:inputs}
For timing measurements, we generate matrices by sampling independent $\mathcal N(0,1)$ entries in \texttt{float32}.
We consider single matrices and batches.
The single matrices have sizes $8192\times8192$, $16384\times16384$,
$4096\times14336$, $14336\times4096$, $8192\times28672$, $28672\times8192$,
$16384\times53248$, and $53248\times16384$.
The batched tensors have sizes $64\times4096\times4096$, $216\times2048\times7168$,
$216\times7168\times2048$, $128\times1024\times8192$, and $128\times8192\times1024$.
The first dimension of each batched tensor is its batch size.

To study how accurately each method approximates the polar factor, we construct matrices
with geometrically spaced singular values and random singular vectors.
These matrices have dimensions
$2048\times2048$, $1024\times3584$, and $3584\times1024$.
For dimensions $n\times m$ equal to $2048\times2048$ or $1024\times3584$, we generate independent $n\times n$ and $m\times n$ Gaussian matrices in \texttt{float64}, with independent $\mathcal N(0,1)$ entries.
A reduced QR decomposition factors each matrix into a matrix with $n$ orthonormal columns and an $n\times n$ upper triangular matrix.
Denote the orthonormal factors by $U\in\mathbb{R}^{n\times n}$ and $V\in\mathbb{R}^{m\times n}$.

We choose a target minimum singular value $\sigma_{\min}\in\{10^{-3},10^{-4}\}$.
For $\gamma\in(0,1)$, we set $\sigma_i$ proportional to $\gamma^{(i-1)/(n-1)}$, $i=1,\ldots,n$, scaled so that $\sum_i\sigma_i^2=1$.
These singular values are geometrically spaced and give a matrix of unit Frobenius norm.
We select $\gamma$ by bisection so that the smallest singular value $\sigma_n$ equals $\sigma_{\min}$.
We form $U\operatorname{diag}(\sigma_1,\ldots,\sigma_n)V^\mathsf{T}$ in \texttt{float64} and convert it to the stored input $X$ in \texttt{float16}.
For each of the dimensions $2048\times2048$ and $1024\times3584$, we generate
500 independent pairs of Gaussian matrices. Each pair determines $U,V$, which are
shared across both values of $\sigma_{\min}$. Every method receives the same stored inputs.
The $3584\times1024$ inputs are transposes of the stored $1024\times3584$ inputs.
In total, this gives 3,000 synthetic inputs, with 500 for each of the 3 dimensions and each of the 2 values of $\sigma_{\min}$.
The \texttt{Polar Switch} circuits are designed for singular values in $[0.001,1]$ (\Cref{app:coefficients}), so the spectra with $\sigma_{\min}=10^{-4}$ contain singular values below this interval.

\paragraph{Reference computation.}
After generating the test matrix $X$, we apply \texttt{Polar Switch} as in \Cref{alg:standard,alg:polar-switch-methods}, \textsf{Jordan}, \textsf{PE}, and \textsf{PE (BF16)} as in \Cref{alg:baselines}, and \textsf{GNS} with its official code.
For tall inputs, the code forms $X^\mathsf{T}X$ and applies polynomial multipliers
on the right. This is equivalent to processing the transpose and transposing the output back.
\Cref{fig:campaign-stability,fig:drift-mixing} show each synthetic spectrum and the Modded-NanoGPT \#91 matrices in separate panels, with equal weight per matrix.

\paragraph{Diagonal drift and directional mixing.}
\label{app:drift}
As in \Cref{sec:synthetic}, $X$ is the input, $X_T$ is the output, $n$ is the smaller dimension of $X$, and $U\Sigma V^\mathsf{T}$ is the compact SVD of $\texttt{float64}(X)$, where $\Sigma=\operatorname{diag}(\sigma_1,\ldots,\sigma_n)$.
Let $p$ be the scalar map of the method, including its normalization.
In exact arithmetic, the output is $Up(\Sigma)V^\mathsf{T}$: it has the same singular vectors as $X$, and each singular value $\sigma_j$ becomes $p(\sigma_j)$.
With $X_T$ converted to \texttt{float64}, let $H_T=U^\mathsf{T}X_TV\in\mathbb{R}^{n\times n}$ and $R_T=X_T-UH_TV^\mathsf{T}$, so that $H_T=p(\Sigma)$ and $R_T=0$ in exact arithmetic.
We define
\begin{align}
 \text{diagonal drift}&=\frac1{\sqrt n}\left(\sum_{i=1}^n\left((H_T)_{ii}-p(\sigma_i)\right)^2\right)^{1/2},
 \label{eq:experimental-diagonal-drift}\\
 \text{directional mixing}&=\frac1{\sqrt n}\left(\sum_{i\ne j}(H_T)_{ij}^2+\|R_T\|_F^2\right)^{1/2}.
 \label{eq:experimental-directional-mixing}
\end{align}
Diagonal drift measures how much rounding changes the singular values along the singular vectors of $X$, and directional mixing measures how much rounding moves the output into other directions (\Cref{fig:drift-mixing}).
Together, they give the total rounding error: by the definition of $R_T$, $X_T-Up(\Sigma)V^\mathsf{T}=U\left(H_T-p(\Sigma)\right)V^\mathsf{T}+R_T$, so
\begin{align*}
 \|X_T-Up(\Sigma)V^\mathsf{T}\|_F^2
 &=\left\|U\left(H_T-p(\Sigma)\right)V^\mathsf{T}\right\|_F^2+2\operatorname{tr}\left(\left(H_T-p(\Sigma)\right)^\mathsf{T}U^\mathsf{T}R_TV\right)+\|R_T\|_F^2\\
 &=\|H_T-p(\Sigma)\|_F^2+\|R_T\|_F^2
  =\sum_{i=1}^n\left((H_T)_{ii}-p(\sigma_i)\right)^2+\sum_{i\ne j}(H_T)_{ij}^2+\|R_T\|_F^2\\
 &=n\left((\text{diagonal drift})^2+(\text{directional mixing})^2\right).
\end{align*}
The first line expands the square and uses the cyclic property of the trace.
The second line uses $U^\mathsf{T}U=V^\mathsf{T}V=I$, $U^\mathsf{T}R_TV=H_T-H_T=0$, and that $p(\Sigma)$ is diagonal, so the off-diagonal entries of $H_T-p(\Sigma)$ are those of $H_T$.

\begin{figure}[htbp]
  \centering
  \includegraphics{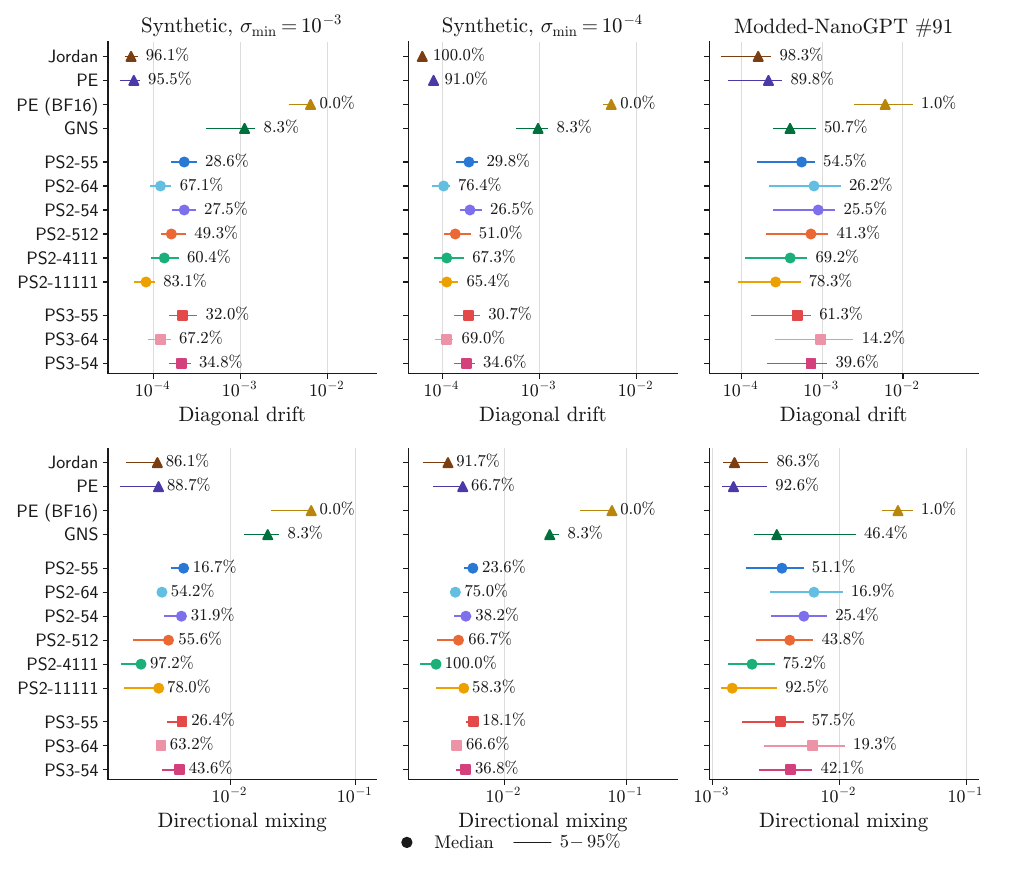}
  \caption{Diagonal drift (top) and directional mixing (bottom), defined in \eqref{eq:experimental-diagonal-drift} and \eqref{eq:experimental-directional-mixing}, with the panels and conventions of \Cref{fig:campaign-stability}.}
  \label{fig:drift-mixing}
\end{figure}

\paragraph{Runtime measurement.}
\label{app:timing}
For each tensor shape, we time all methods in five runs.
Each run starts a new Python process, uses the same input for every method, and compiles each method with \texttt{torch.compile} before timing.
The calls of different methods are interleaved: each round calls every method once.
After warm-up calls, each method has at least 30 timed calls per run for single matrices and at least 15 for batches.
We time each call with CUDA events, so a runtime covers one orthogonalization call on a single matrix or on a whole batch.
The runtime of a method is the median time of its calls in a run, and we report the median of these runtimes over the five runs.
For a runtime ratio, we divide the time of each call by the time of the reference call in the same round, where the reference method is \textsf{PE}, and also \textsf{GNS} for rectangular inputs, and we take the same two medians.
\Cref{fig:fp16-timing-ranges-combined} shows the ratios relative to \textsf{PE} for square (top), tall (middle), and wide (bottom) inputs, and its bars span the minimum and maximum of each ratio over the five runs.
\begin{figure}[htbp]
\centering
\includegraphics{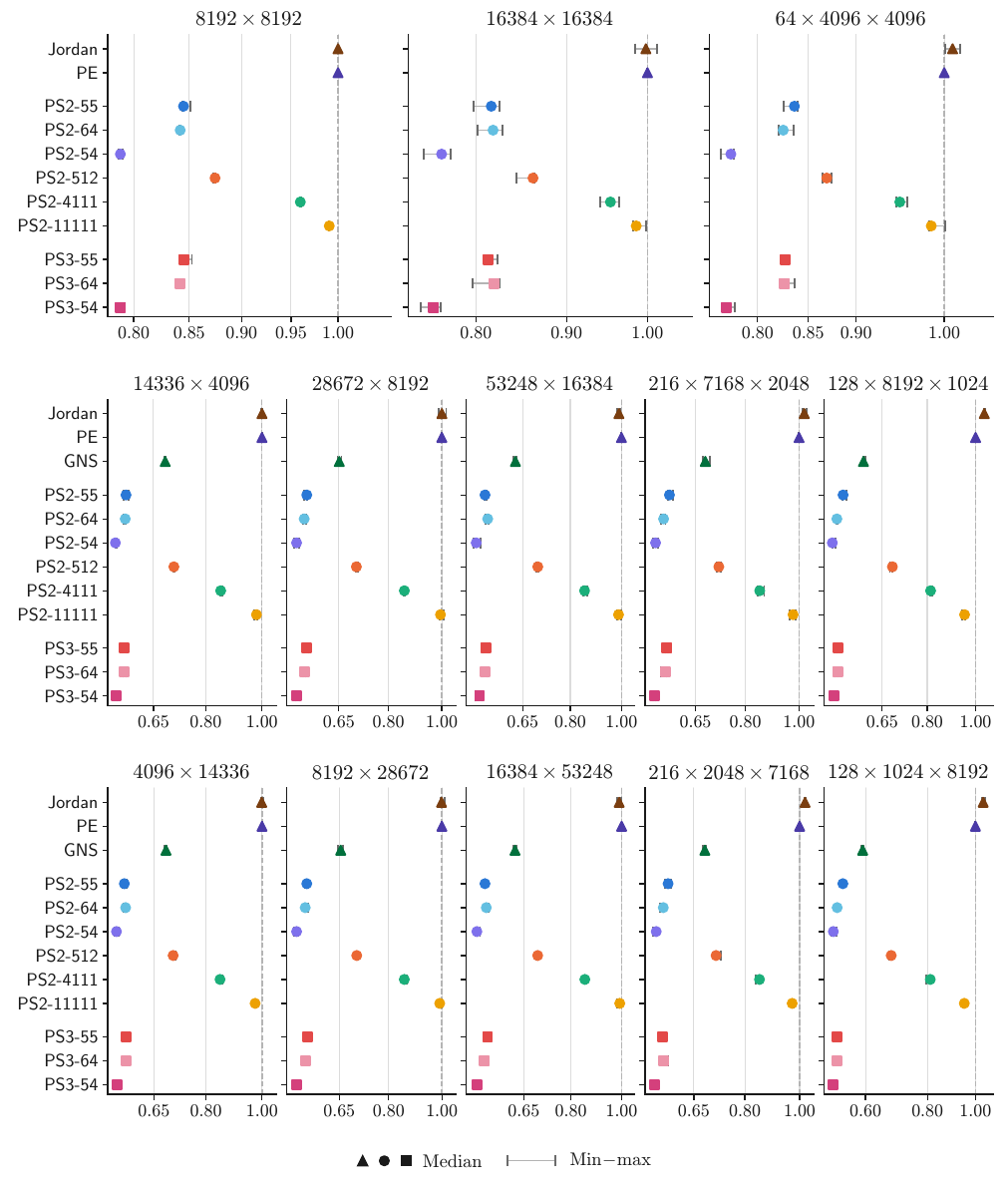}
\caption{Runtime ratios relative to \textsf{PE}.
Values below one indicate lower runtime.
\textsf{GNS} appears only for rectangular shapes, because it falls back to \textsf{PE} for square matrices.
Markers show medians across five runs, and bars span the minimum and maximum.}
\label{fig:fp16-timing-ranges-combined}
\end{figure}

\end{document}